\documentclass[reqno]{amsart}

\usepackage{general,resolvent-kernel,txfonts,cancel} %
\usepackage[bookmarks, colorlinks=true, pdfstartview=FitV, linkcolor=blue, citecolor=blue, urlcolor=blue]{hyperref}

\numberwithin{equation}{section} \numberwithin{theorem}{section}

\begin{document}


\title{The resolvent kernel for PCF self-similar fractals}

\author[Ionescu]{Marius Ionescu}
\address{Cornell University, Ithaca, NY 14850-4201 USA}
\email{mionescu@math.cornell.edu}

\author[Pearse]{Erin P. J. Pearse}
\address{University of Iowa, Iowa City, IA 52246-1419 USA}
\email{erin-pearse@uiowa.edu}
\thanks{The work of EPJP was partially supported by the University of Iowa Department of Mathematics NSF VIGRE grant DMS-0602242.}

\author[Rogers]{Luke G. Rogers}
\address{University of Connecticut, Storrs, CT 06269-3009 USA}
\email{rogers@math.uconn.edu}

\author[Ruan]{Huo-Jun Ruan}
\address{Zhejiang University, Hangzhou, 310027, China, and Cornell University, Ithaca, NY 14850-4201 USA}
\email{ruanhj@zju.edu.cn} 
\thanks{The work of HJR was partially supported by grant NSFC 10601049, and by the Future Academic Star project of Zhejiang University.}

\author[Strichartz]{Robert S. Strichartz}
\address{Cornell University, Ithaca, NY 14850-4201 USA}
\email{str@math.cornell.edu}
\thanks{The work of RSS was partially supported by NSF grant DMS 0652440.}

\begin{abstract}
  For the Laplacian $\Delta$ defined on a p.c.f. self-similar fractal, we give an explicit formula for the resolvent kernel of the Laplacian with Dirichlet boundary conditions, and also with Neumann boundary conditions. That is, we construct a symmetric function $G^{(\lambda)}$ which solves $(\lambda \mathbb{I} - \Delta)^{-1} f(x) = \int G^{(\lambda)}(x,y) f(y) \, d\mu(y)$. The method is similar to Kigami's construction of the Green kernel in \cite[\S3.5]{Kig01} and is expressed as a sum of scaled and ``translated'' copies of a certain function $\psi^{(\lambda)}$ which may be considered as a fundamental solution of the resolvent equation. Examples of the explicit resolvent kernel formula are given for the unit interval, standard Sierpinski gasket, and the level-3 Sierpinski gasket $SG_3$. 
\end{abstract}



\date{\bf\today}


\maketitle

\setcounter{tocdepth}{1} \tableofcontents

\allowdisplaybreaks

\section{Introduction}

A theory of analysis on certain self-similar fractals is developed around the Laplace operator \Lap in \cite{Kig01}. In this paper, we consider the resolvent function $(\gl \id - \Lap)^{-1}$ and obtain a kernel for this function when the Laplacian is taken to have Dirichlet or Neumann boundary conditions. That is, we construct a symmetric function $G\parlam$ which weakly solves $(\gl \id - \Lap) G\parlam(x,y) = \gd(x,y)$, meaning that
\linenopar
\begin{equation}\label{eqn:resolvent-kernel-intro}
  \int G\parlam(x,y) f(y) \, d\gm(y) = (\gl \id - \Lap)^{-1} f(x).
\end{equation}
For the case $\gl=0$, this is just the Green function for \Lap. Consequently, it is not surprising that our construction is quite analogous to that of the Green function as carried out in \cite[\S3.5]{Kig01}; see also \cite[\S2.6]{Str06} for the case of the Sierpinski gasket (and the unit interval) worked out in detail, and \cite{Kig03}.

We present our main results in \S\ref{sec:preview-of-results}, just after the introduction of the necessary technical terms in \S\ref{sec:background,notation,and-fundamentals}.
It is the authors' hopes that the resolvent kernel will provide an alternate route to obtaining heat kernel estimates (see \cite{FitzHambKuma94,HamKuma99}) in this setting, as well as other information about spectral operators of the form
\linenopar
\begin{align}\label{eqn:general-spectral-operator}
  \gx(\Lap) = \int_\gG \gx(\gl) (\gl\id-\Lap)^{-1} \, d\gl,
\end{align}
in the same manner as used by Seeley \cite{See67,See69} for the Euclidean situation. Some initial results in this direction will appear in \cite{Rogers:ResolventKernelII}.

To explain the method of construction for the resolvent kernel, we carry out the procedure in the case of the unit interval in \S\ref{sec:resolvent-kernel-on-the-unit-interval}; we believe this particular method has not previously appeared in the literature. In \S\ref{sec:Dirichlet-resolvent-kernel}, we show how the construction may be generalized to any post-critically finite self-similar fractal.
In \S\ref{sec:example-sg}, we give the explicit formulas for the Sierpinski gasket and in \S\ref{sec:example-sg3} we give the explicit formulas for a variant of the Sierpinski gasket which we call $S\negsp[2]G_3$. 

\subsubsection*{Acknowledgements} The authors are indebted to the referee for many keen observations and a scrupulously detailed report.

\subsection{Background, notation, and fundamentals}
\label{sec:background,notation,and-fundamentals}

We work in the context of post-critically finite (p.c.f.) self-similar fractals. The full and precise definition may be found in \cite[Def.~1.3.13]{Kig01}, but for the present context it suffices to think of such objects as fractals which may be approximated by a sequence of graphs, via an iterated function system (IFS). A more general setting is possible; cf. \cite{Kig03}. We now make this more precise.

\begin{defn}\label{def:IFS}
  Let $\{F_1,F_2,\dots,F_J\}$ be a collection of Lipschitz continuous functions on \bRd with $0 < Lip(F_j) < 1$ for each $j$. Let $X$ denote the \emph{attractor} of this IFS, that is, $X$ is the unique nonempty and compact fixed point of the set mapping $F(A) := \bigcup_{j=1}^J F_j(A)$. The set $X$ is frequently also called the \emph{self-similar set} associated to this IFS; existence and uniqueness of $X$ was shown in \cite{Hut81}. 
\end{defn}

From the IFS introduced in the previous definition, we now build a sequence of graphs which approximates $X$ in a suitable sense.

\begin{defn}\label{def:words-on-1,2,...,J}
  We use $\word = \word_1 \word_2 \cdots \word_m$ to denote a \emph{word} of length $|\word|=m$ on the symbol alphabet $\{1,2,\dots,J\}$. This notation is used to denote a composition of the mappings $F_j$ via $F_\word = F_{\word_1} \comp F_{\word_2} \comp \dots \comp F_{\word_m}$. Similarly, $K_\word = F_\word(X)$ refers to a certain \emph{$m$-level cell}. The collection of all finite words is denoted $W_\ast := \bigcup_m \{1,2,\dots,J\}^m$.
\end{defn}

\begin{defn}\label{def:graph-approximants}
  Each map $F_j$ of the IFS defining $X$ has a unique fixed point $x_j$. The \emph{boundary} of $X$ is the largest subset $V_0 \ci \{x_1,\dots, x_N\}$ satisfying
  \[F_\word(X) \cap F_{\word'}(X) \ci F_\word(V_0) \cap F_{\word'}(V_0), \q
    \text{ for any $\word \neq \word'$ with $|\word |=|\word'|$.}\]
  The p.c.f. condition mentioned above means that cells $F_j(X)$ intersect only at points of $F_j(V_0)$. The \emph{boundary of an $m$-cell} is $\bdy K_\word := F_\word(V_0)$.

  Let $G_0$ be the complete graph on $V_0$, and inductively define $G_m := F(G_{m-1})$.   Also, we use the notation $x \nbr[m] y$ to indicate that $x$ and $y$ are \emph{$m$-level neighbours}, i.e., that there is an edge in $G_m$ with endpoints $x$ and $y$.
  We use $V_m = F^m(V_0)$ to denote the vertices of $G_m$, and $V_\ast := \bigcup_m V_m$. The fractal $X$ is the closure of $V_\ast$ with respect to either the Euclidean or resistance metric. A discussion of the resistance metric may be found in \cite[\S1.6]{Str06} or \cite[\S2.3]{Kig01}.
\end{defn}

Now we are able to make precise the sense in which $X$ is the limit of graphs: one may compute the Laplacian (and other analytic objects, including graph energy, resistance distance, etc.) for functions $u:X \to \bR$ by computing it on $G_m$ and taking the limit as $m \to \iy$.

\begin{defn}\label{def:graph-energy}
  We assume the existence of a \emph{self-similar (Dirichlet) energy form} \energy on $X$. That is, for functions $u:X\to\bR$, one has
  \linenopar
  \begin{equation}\label{eqn:def:energy-form}
    \energy(u) = \sum_{j=1}^J r_j^{-1} \energy(u \comp F_j),
  \end{equation}
  for some choice of \emph{renormalization factors} $r_1,\dots,r_J \in (0,1)$ depending on the IFS. This quadratic form is obtained from the approximating graphs as the appropriately renormalized limit of $\energy_{G_m}(u) := \energy_{G_m}(u,u)$, where the $m$-level bilinear form is defined
  \linenopar
  \begin{equation}\label{eqn:mth-graph-energy}
    \energy_{G_m}(u,v) := \frac12 \sum_{\substack{x,y \in V_m \\ x \mnbr y}} c_{xy} (u(x)-u(y))(v(x)-v(y)).
  \end{equation}
  The constant $c_{xy}=c^{(m)}_{xy}$ refers to the \emph{conductance} of the edge in $G_m$ connecting $x$ to $y$ (with $c_{xy}=0$ if there is no such edge). The dependence of $c^{(m)}_{xy}$ on $m$ is typically suppressed, as $x \mnbr y$ for at most one value of $m$ on p.c.f. fractals. The \emph{domain} of \energy is
  \linenopar
  \begin{align*}\label{eqn:def:domain-of-energy}
    \dom \energy := \{u:X \to \bR \suth \energy(u) < \iy\}.
  \end{align*}
\end{defn}

\begin{defn}\label{def:self-similar-measure}
  We also assume the existence of a \emph{self-similar measure} \gm
  \linenopar
  \begin{equation}\label{eqn:def:measure}
    \gm(A)=\sum_{j=1}^J \gm_j \gm(F_j^{-1}(A)),
  \end{equation}
  with weights $\gm_j$ satisfying $0<\mu_j<1$ and $\sum_j \gm_j = 1$, and normalized so that $\gm(X)=1$. With the notation of Definition~\ref{def:words-on-1,2,...,J}, the measure of the $m$-cell $K_\word$ is denoted by $\gm(K_\word) = \gm_\word := \gm_{\word_1}\gm_{\word_2}\dots\gm_{\word_m}$.
  The \emph{standard measure} refers to the case $\gm_j = \frac1J$, for each $j$.
\end{defn}

\begin{remark}\label{rem:r_j-vs-Lip(F_j)-vs-mu_j}
  The renormalization factor $r_j$ should be confused neither with the contraction factors $Lip(F_j)$ of the maps of the IFS, nor the weights $\gm_j$ of the self-similar measure \gm. The values of these constants are completely independent.

  Also, it should be noted that the existence of a self-similar energy asserted in Definition~\ref{def:graph-energy} is a strong assumption. While the the self-similar measures of Definition~\ref{def:self-similar-measure} always exist \cite{Hut81}, the existence of the self-similar energy is a much more delicate question; cf. \cite{Sabot97}.
\end{remark}

\begin{defn}\label{def:Laplacian}
  The \emph{Laplacian} is defined weakly in terms of the energy form. For $u \in \dom \energy$ and $f$ continuous, one says $u \in \dom \Lap$ with $\Lap u = f$ iff
  \linenopar
  \begin{equation}\label{eqn:def:Laplacian}
    \energy(u,v)
    = - \int_X f v \, d\gm, \q \textrm{for all } v \in \dom_0 \energy,
  \end{equation}
  where $\dom_0 \energy$ is the set of functions in $\dom\energy$ which vanish on $\bdy X = V_0$. Note that the Laplacian depends on the choice of measure \gm.

More generally, if \eqref{eqn:def:Laplacian} holds with $f \in L^2(d\gm)$, then one says $u \in \dom_{L^2} \Lap$; and if 
  \linenopar
  \begin{equation}\label{eqn:def:measure-Laplacian}
    \energy(u,v)
    = - \int_X v \, d\gm, \q \textrm{for all } v \in \dom_0 \energy,
  \end{equation}
  for a finite signed measure \gm with no atoms, then one says $u \in \dom_{\sM} \Lap$.
\end{defn}

It follows from \eqref{eqn:def:energy-form}, \eqref{eqn:def:measure} and Definition~\ref{def:Laplacian} that $\Lap$ satisfies the scaling identity
\linenopar
\begin{equation}\label{eqn:Laplacian-scaling}
  \Lap(u\comp F_j) = r_j \gm_j (\Lap u) \comp F_j,
\end{equation}
and pointwise formula given by the uniform limit
\linenopar
\begin{equation}\label{eqn:Laplacian-pointwise}
  \Lap u(x)
  = \lim_{m\to\iy}\left(\int_X h_x^{(m)} \,d\gm\right)^{-1}\Lap_m u(x),
  \q \text{for } x \in V_\ast \less V_0,
\end{equation}
where $h_x^{(m)}$ is a piecewise harmonic spline satisfying $h_x^{(m)}(y) = \gd_{xy}$ for $y\in V_m$, and
\linenopar
\begin{equation}\label{eqn:def:h_x-pw-harmonic-spline}
  \Lap_m u(x) = \sum_{y \mnbr x} c_{xy}(u(y)-u(x)),
  \q \text{for } x \in V_m.
\end{equation}
Roughly speaking, $h_x^{(m)}$ is a ``tent'' function with peak at $x$ which vanishes outside the $m$-cell containing $x$. See \cite[\S2.1--\S2.2]{Str06} for details.

\begin{defn}\label{def:normal-derivative}
  The \emph{normal derivative} of a function $u$ is computed at a boundary point $q \in V_0$ by
  \linenopar
  \begin{equation}\label{eqn:def:normal-derivative-level0}
    \dn u(q)
    := \lim_{m \to \iy} \frac1{r_i^{m}} \sum_{y \mnbr q} (u(q)-u(y)),
    \qq q \in V_0.
  \end{equation}
  At a general junction point $x=F_\word q$, the normal derivative is computed with respect to a specific $m$-cell $K_\word$:
  \linenopar
  \begin{equation}\label{eqn:def:normal-derivative}
    \dn[K_\word] u(x) = \dn[K_\word] u(F_\word q)
    := \frac1{r_{\word_1} \cdots r_{\word_m}} \dn (u \comp F_\word)(q).
  \end{equation}
\end{defn}

\subsection{Statement of main result}
\label{sec:preview-of-results}

\begin{theorem}\label{thm:resolvent-formula-preview}
  Assume that \gl is not a Dirichlet eigenvalue of \Lap, and neither is $r_\word \gm_\word \gl$, for any $\word \in W_\ast$. For the Laplacian on $X$ with Dirichlet boundary conditions, the resolvent kernel $G\parlam$ defined by \eqref{eqn:resolvent-kernel-intro} is given by the formula
  \linenopar
  \begin{align}
    G\parlam(x,y)
    =& \sum_{\word \in W_\ast} r_\word \gY\parlam[r_\word \gm_\word \gl] (F_\word^{-1} x, F_\word^{-1}y),
      \label{eqn:resolvent-formula-preview} \\
    \text{where }
    \gY\parlam(x,y)
    :=& \negsp[8] \sum_{p,q \in V_1 \less V_0} \negsp[8] G\parlaq{pq} \gy\parlaq{p}(x) \gy\parlaq{q}(y).
    \label{eqn:Psi-preview}
  \end{align}
  where convention stipulates $\gY\parlam[r_\word \gm_\word \gl] (F_\word^{-1} x, F_\word^{-1}y) = 0$ for $x,y$ not in $F_\word X$.
  In formula \eqref{eqn:Psi-preview}, $\gy\parlaq{p}$ is the solution to the resolvent equation at level 1, i.e.
  \linenopar
  \begin{equation}\label{eqn:basic-resolvent-solution-preview}
    \begin{cases}
      (\gl \id - \Lap)\gy\parlaq{p} = 0, &\text{on each } K_j=F_j(X), \\
      \gy\parlaq{p}(q) = \gd_{pq}, &\text{for $p \in V_1 \less V_0$ and $q \in V_1$},
    \end{cases}
  \end{equation}
  where $\gd_{pq}$ is the Kronecker delta. The coefficients $G\parlaq{pq}$ in \eqref{eqn:Psi-preview} arise as the entries of the inverse of the matrix $B$ given by
  \linenopar
  \begin{equation}\label{eqn:Bobmatrix-preview}
    B\parlaq{pq} := \sum_{K_j \ni q} \dn[K_j] \gy\parlaq{p}(q),
    \qq q \in F_j(V_0),
  \end{equation}
  where the sum is taken over all 1-cells containing $q$.
\end{theorem}

This result appears with proof as Theorem~\ref{thm:Dirichlet-resolvent-formula}; a similar formula for Neumann boundary conditions appears in Theorem~\ref{thm:Neumann-resolvent-formula}.

\begin{remark}\label{rem:sum-over-sets-notation}
  In \eqref{eqn:Bobmatrix-preview} and elsewhere, we use the notation $\sum_{K_j \ni q}$ to indicate a sum being taken over the set $\{j \suth q \in K_{j} = F_j(X)\}$.
\end{remark}

The rationale for the definitions  \eqref{eqn:resolvent-formula-preview}--\eqref{eqn:Bobmatrix-preview} is best explained by the following heuristic argument and by comparison to \cite[Thm.~2.6.1]{Str06}. One would like $\gY\parlam$ to be a weak solution to the resolvent equation on a 1-cell $C = F_i(X)$, except at the boundary where some Dirac masses may appear. However, this implies that $r_i \gY\parlam[r_i \gm_i \gl](F_i^{-1}x,F_i^{-1}y)$ will be a weak solution on the 2-cell $F_i(C)$, and in the limit \eqref{eqn:resolvent-formula-preview} gives a solution on the entire fractal. Each term added to the partial sum of \eqref{eqn:resolvent-formula-preview} corresponds to canceling the Dirac masses at the previous stage and introducing new ones at the next; these are wiped away in the limit.

For $\gY\parlam$ to be a weak solution at level $1$, we mean that if $u \in \dom \Lap$ and $u$ vanishes on $\bdy X = V_0$, then
\linenopar
\begin{align*}
  \int_X \gY\parlam(x,y) (\gl\id-\Lap) u(y) \,d\gm(y) = \sum_{p \in V_1 \less V_0} \gy \parlaq{p}(x) u(p).
\end{align*}
With \eqref{eqn:Psi-preview} as given above, integration by parts and linearity give
\linenopar
\begin{align*}
  \int_X \gY\parlam(x,y) (\gl\id-\Lap) u(y) \,d\gm(y)
  &= \int_X \left[(\gl\id-\Lap_y) \gY\parlam(x,y)\right] u(y) \,d\gm(y) \\
  &= \sum_{p,q \in V_1 \less V_0} G\parlaq{pq} \gy\parlaq{p}(x) \int_X\left[(\gl\id-\Lap) \gy\parlaq{q}(y)\right] u(y) \,d\gm(y),
\end{align*}
where we used the notation $\Lap_y$ to indicate that the operator \Lap is applied with respect to the variable $y$.

Now by \eqref{eqn:basic-resolvent-solution-preview}, $\gy\parlaq{q}$ satisfies the resolvent equation on the interior of the 1-cells, but $-\Lap \gy\parlaq{q}$ has Dirac masses at the boundary points with weights $B\parlaq{qs} := \sum_{K_j \ni s} \dn[K_j] \gy\parlaq{q}(s)$. In other words, 
we have $\Lap \gy\parlaq{q} = \gl\gy\parlaq{q}$ except on $V_1 \setminus V_0$, so that $(\gl\id-\Lap) \gy\parlaq{q}(y) = \sum_{s \in V_1 \less V_0} B\parlaq{qs} \gd_s(y)$, where $\gd_s$ is the Dirac mass at $s$. Therefore, the calculation above may be continued:
\linenopar
\begin{align*}
  \int_X \gY\parlam(x,y) (\gl\id-\Lap) u(y) \,d\gm(y)
  &= \negsp[8] \sum_{p,q,s \in V_1 \less V_0} \negsp[8] \gy\parlaq{p}(x) G\parlaq{pq} B\parlaq{qs} \int_X \gd_s(y) u(y) \,d\gm(y)\\
  &= \sum_{p \in V_1 \less V_0} \gy\parlaq{p}(x) u(p).
\end{align*}

The foregoing computation is the origin and motivation for \eqref{eqn:Psi-preview}--\eqref{eqn:Bobmatrix-preview}. A key technical point is the use of a linear combination $u$ of vectors $\gy\parlaq{q}$ for which $(\gl\id-\Lap)u$ is a \emph{single} (weighted) Dirac mass at $p$. From the calculation, it is clear that this hinges on the invertibility of $B$; this is the significance of Lemma~\ref{thm:B-is-invertible}.

As mentioned just above, once the solution is obtained on level 1, it may be transferred to a cell $F_\word(X)$ by rescaling appropriately. However, this is not sufficient to allow us to compute $(\gl\id-\Lap_{y}) G\parlam(x,y)$; some finesse is required to ensure that these solutions match where these cells intersect, that is, on the boundary points $V_{m+1}\setminus V_0$. Some further work is needed; this is carried out in the technical lemmas of \S\ref{sec:Dirichlet-resolvent-kernel}.

\section{The resolvent kernel for the unit interval}
\label{sec:resolvent-kernel-on-the-unit-interval}

The unit interval $I=[0,1]$ has a self-similar structure derived from the IFS consisting of $F_1(x) = \frac x2$ and $F_2(x) = \frac x2 + \frac12$. In this section, we exploit this perspective to derive the resolvent kernel for the Dirichlet Laplacian on $I$ by mimicking the construction of the Green function in \cite[\S3.5]{Kig01} (see also \cite[\S2.6]{Str06}). This exposition is intended to make the general case (presented in the next section) easier to digest. We build towards the result stated formally in Prop.~\ref{thm:resolvent-formula-interval-preview}.

\begin{prop}\label{thm:resolvent-formula-interval-preview}
  Let $\Lap = \frac{d^2}{dx^2}$ be the Laplacian on the unit interval $I=[0,1]$, taken with Dirichlet boundary conditions.
  If \gl is not a Dirichlet eigenvalue of \Lap, then the resolvent kernel $G\parlam$ in \eqref{eqn:resolvent-kernel-intro} is given by
  \linenopar
  \begin{align}\label{eqn:resolvent-formula-interval-preview}
    G\parlam(x,y)
    =& \sum_{m=0}^\iy \sum_{|\word|=m} \frac1{2^m} \gY\parlam[\gl/4^m] (F_\word^{-1} x, F_\word^{-1}y), \\
    \text{for }\q
    \gY\parlam(x,y)
    :=& \frac{\sinh \frac{\sqrt\gl}2}{2\sqrt\gl \cosh \frac{\sqrt\gl}2} \gy\parlam(x) \gy\parlam(y), \label{eqn:Psi-interval-preview} \\
    \text{and }\q
    \gy\parlam(x) 
    :=& \frac{1}{\sinh \frac{\sqrt{\gl}}{2}}
    \begin{cases}
      \sinh \sqrt{\gl} x, & x\leq \frac12, \\
      \sinh \sqrt{\gl} (1-x),& x\geq \frac12,
    \end{cases}
    \label{eqn:psi-interval-preview}
  \end{align}
  where convention stipulates $\gY\parlam[\gl/4^m] (F_\word^{-1} x, F_\word^{-1}y) = 0$ for $x,y$ not in $F_\word I$.
\end{prop}

\begin{remark}[A preview of the general case]
  \label{rem:preview-of-general-formulation}
  Note that the sum in \eqref{eqn:resolvent-formula-interval-preview} is finite if $x \neq y$, or if $x=y$ is dyadic. More importantly, $\gy\parlam = \gy\parlaq{1/2}$ is the solution to the resolvent equation at level 1, i.e.
  \linenopar
  \begin{equation}\label{eqn:basic-resolvent-interval-solution}
    \begin{cases}
      (\gl \id - \Lap)\gy\parlam = 0, &\text{on } (0,\tfrac12) \text{ and } (\tfrac12,1), \\
      \gy\parlam(0) = \gy\parlam(1) = 0, &\text{and } \gy\parlam(\tfrac12) = 1.
    \end{cases}
  \end{equation}
  In \S\ref{sec:Dirichlet-resolvent-kernel}, we develop the resolvent kernel in the general case from these observations.
\end{remark}

In keeping with the self-similar spirit of the sequel, we use the term \emph{1-cell} in reference to the subintervals $[0,\tfrac12]$ and $[\tfrac12,1]$ in the following proof.

  \begin{figure}
    \centering
    \scalebox{0.90}{\includegraphics{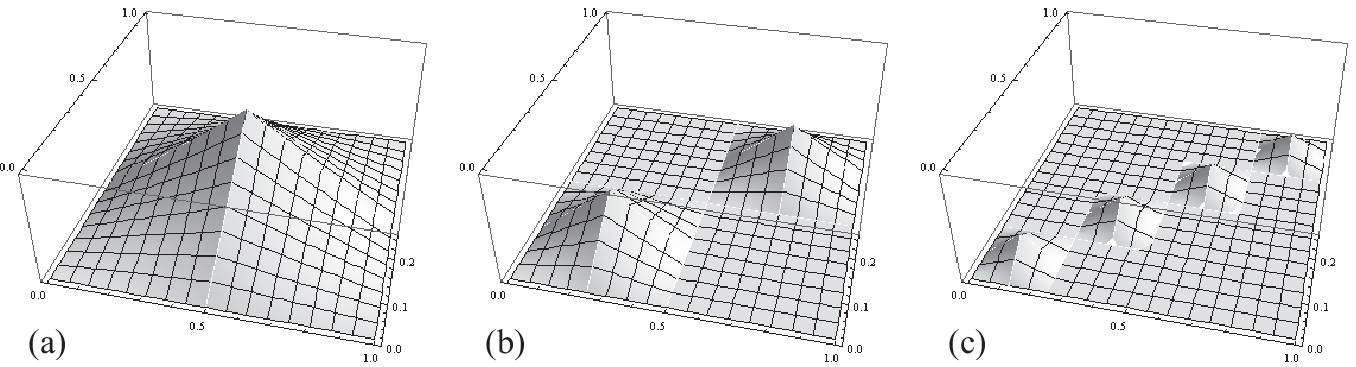}}
    \caption{\captionsize Mathematica plot of $\gY\parlam$ from Prop.~\ref{eqn:resolvent-formula-interval-preview} for $\gl=1$. 
    (a) $\gY\parlam(x,y)$.
    (b) $\frac12(\gY\parlam[\gl/4](2x,2y)+\gY\parlam[\gl/4](2x-1,2y-1))$. 
    (c) $\frac14(\gY\parlam[\gl/16](4x,4y)+\gY\parlam[\gl/16](4x-1,4y-1) +  \gY\parlam[\gl/16](4x-2,4y-2) + \gY\parlam[\gl/16](4x-3,4y-3))$.}
    \label{fig:interval-resolvent}
  \end{figure}

  \begin{figure}
    \centering
    \scalebox{0.85}{\includegraphics{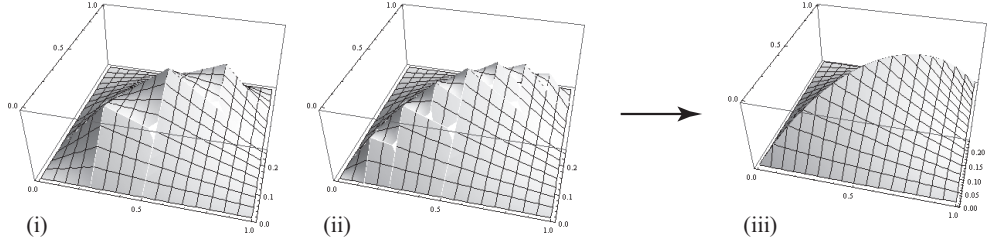}}
    \caption{\captionsize Mathematica plot of $G\parlam$ for $\gl=1$ and two of its partial sums.
    (i) The sum of (a) and (b) in Fig.~\ref{fig:interval-resolvent}.
    (ii) The sum of (a), (b), (c) in Fig.~\ref{fig:interval-resolvent}.
    (iii) The resolvent kernel $G\parlam(x,y)$ obtained in the limit.}
    \label{fig:interval-resolvent-stages}
  \end{figure}

\begin{proof}[Proof of Prop.~\ref{thm:resolvent-formula-interval-preview}]
  On the unit interval $I$, one has the resolvent kernel
  \linenopar
  \begin{equation}
    G\parlam(x,y) = \frac1{\sqrt\gl \sinh\sqrt\gl}
    \begin{cases}
      \sinh \sqrt{\gl}(1-y) \sinh \sqrt{\gl} x & x\leq y,  \\
      \sinh \sqrt{\gl}y \sinh \sqrt{\gl} (1-x), & x\geq y.
    \end{cases}
    \label{eqn:resolvent-interval-case x<y}
  \end{equation}
  For $x\leq \frac12 \leq y$, one has
  \linenopar
  \begin{align}
    G\parlam(x,y)
    &= \dfrac{\sinh \sqrt{\gl}(1-y) \sinh \sqrt{\gl} x}{\sqrt{\gl} \sinh{\sqrt{\gl}}}
      &&\text{by \eqref{eqn:resolvent-interval-case x<y}} \notag \\
    &= \frac{\sinh \frac{\sqrt{\gl}}{2}}{2\sqrt{\gl}\cosh \frac{\sqrt{\gl}}{2}} \cdot \frac{\sinh\sqrt{\gl}x \sinh \sqrt{\gl}(1-y)}{\sinh^2 \frac{\sqrt{\gl}}{2}}
      &&\sinh 2a = 2 \sinh a \cosh a \notag \\
    &= \frac{\sinh \frac{\sqrt{\gl}}{2}}{2\sqrt{\gl}\cosh \frac{\sqrt{\gl}}{2}} \gy\parlam(x) \gy\parlam(y)
      &&\text{by \eqref{eqn:psi-interval-preview}.}
      \label{eqn:resolvent-interval-diff-cells}
  \end{align}
  The same computation can be repeated for $y \leq \frac12 \leq x$ and hence \eqref{eqn:resolvent-interval-diff-cells} holds whenever $x$ and $y$ are in different 1-cells of $I$.

  It remains to consider the case when both $x$ and $y$ lie in the same 1-cell of $I$. Suppose that $x \leq y \leq \frac12$ and consider the difference
  \linenopar
  \begin{align}
    R(x,y)
    &:= G\parlam(x,y) - \frac{\sinh \frac{\sqrt{\gl}}{2}}{2\sqrt{\gl}\cosh \frac{\sqrt{\gl}}{2}} \gy\parlam(x)\gy\parlam(y) \notag \\
    &\,= \frac{\sinh \sqrt{\gl}x (\sinh\sqrt{\gl}(1-y)-\sinh\sqrt{\gl}y)} {\sqrt{\gl}\sinh\sqrt{\gl}}
      \label{eqn:resolvent-interval-calc} \\
    &\,= \frac{\sinh \sqrt{\gl}x\sinh{\sqrt{\gl}(\frac12-y)}}{\sqrt{\gl}\sinh\frac{\sqrt{\gl}}{2}} \notag \\
    &\,= \tfrac12 G\parlam[\gl/4](2x,2y),
      \label{eqn:resolvent-interval-x,y<1/2}
  \end{align}
  where \eqref{eqn:resolvent-interval-calc} follows by \eqref{eqn:resolvent-interval-case x<y} and the identity $\sinh(1-a)-\sinh a = 2 \sinh (\tfrac12-a) \cosh \tfrac12$.
  In the case when $y \leq x \leq \frac12$, one also obtains $R(x,y) = \tfrac12 G\parlam[\gl/4](2x,2y)$.
  On the other hand, when $x$ and $y$ are both in the other 1-cell, one obtains (by analogous computations) that $R(x,y) = \tfrac12 G\parlam[\gl/4](2x-1,2y-1)$. Note that if \gl is not a Dirichlet eigenvalue of \Lap, then neither is $\gl/4^m$ for any $m=0,1,2,\dots$.
  Consequently, if we define $\gY\parlam(x,y)$ as in \eqref{eqn:Psi-interval-preview}, then formula \eqref{eqn:resolvent-formula-interval-preview} for $G\parlam(x,y)$ follows.
\end{proof}

\begin{remark}\label{rem:Dirichlet-coefficient-structure}
  It is interesting to note that the coefficient which appears in \eqref{eqn:resolvent-interval-diff-cells} is
  \linenopar
  \begin{align*}
    \frac{\sinh \frac{\sqrt{\gl}}{2}}{2\sqrt{\gl}\cosh \frac{\sqrt{\gl}}{2}}
    = \frac{1}{{\gy\parlam}'(\frac12-) -{\gy\parlam}'(\frac12+) }.
  \end{align*}
  Formally, this indicates $(\gl\id-\Lap) G\parlam(x,y) = \gd(x-y)$; compare to \cite[(2.6.3)]{Str06}.
  Also, observe that
  \linenopar
  \begin{equation*}
    G\parlam(x,\tfrac12)
    = \frac{\sinh \frac{\sqrt\gl}2}{2\sqrt\gl \cosh \frac{\sqrt\gl}2} \gy\parlam(x) \gy\parlam(\tfrac12).
  \end{equation*}
\end{remark}
At each successive iteration of \eqref{eqn:resolvent-interval-x,y<1/2}, one is essentially ``correcting'' the formula on the diagonal for the $m$-cell with rescaled copies of the formula for the $(m+1)$-cell; Figures~\ref{fig:interval-resolvent} and \ref{fig:interval-resolvent-stages} are intended to explain this. In the next section, we follow this strategy for the construction of the resolvent kernel in the general case.

\begin{remark}\label{rem:Neumann-kernel}
  The procedure in the proof of Proposition~\ref{thm:resolvent-formula-interval-preview} may also be carried out for the Neumann case: define a function $\gf\parlam$ to be the solution of
  \linenopar
  \begin{align*}
    \begin{cases}
      (\gl \id -\Lap) \gf\parlam = 0, &\text{on } [0,\frac12] \text{ and } [\frac12,1] \\
      \tfrac{d}{dx}{\gf\parlam}(x) = 0, &x=0,1 \\
      \gf\parlam(\frac12)=1, \vstr
    \end{cases}
  \end{align*}
  which is given by
  \linenopar
  \begin{equation*}
    \gf\parlam (x)
    = \frac{1}{\cosh \frac{\sqrt{\gl}}{2}}
      \begin{cases}
        \cosh\sqrt{\gl}x,      & x\leq \frac12,\\
        \cosh\sqrt{\gl}(1-x),  & x\geq \frac12.
      \end{cases}
  \end{equation*}

  Observe that in parallel to Remark~\ref{rem:Dirichlet-coefficient-structure}, one again has
  \linenopar
  \begin{equation*}
    G\parlam_N\left(x,\tfrac12\right)
    = \frac{\cosh \frac{\sqrt{\gl}}{2}}{2\sqrt{\gl}\sinh \frac{\sqrt{\gl}}{2}} \mspace{3mu} \gf\parlam\negsp[2] (x) \mspace{3mu} \gf\parlam \negsp[5] \left(\tfrac12\right)
  \end{equation*}
  and
  \linenopar
  \begin{equation*}
    \frac{\cosh \frac{\sqrt{\gl}}{2}}{2\sqrt{\gl}\sinh \frac{\sqrt{\gl}}{2}}
    = \frac{1}{\tfrac{d}{dx}\gf\parlam(\frac12-) -{\tfrac{d}{dx}\gf\parlam}(\frac12+) }\,.
  \end{equation*}

  By analogous computations, if we define
  \linenopar
  \begin{equation*}
    \gF_N\parlam(x,y)
    = \frac{\cosh \frac{\sqrt{\gl}}{2}}{2\sqrt{\gl}\sinh \frac{\sqrt{\gl}}{2}} \mspace{3mu} \gf\parlam\negsp[2] (x) \mspace{3mu} \gf\parlam\negsp[3](y),
  \end{equation*}
  then we obtain the Neumann resolvent kernel
  \linenopar
  \begin{equation*}
    G\parlam_N(x,y)
    = \sum_{m=0}^\iy \sum_{|\word|=m} \frac{1}{2^m} \gF\parlam[\gl/4^m](F_\word^{-1}x,F_\word^{-1}y ).
  \end{equation*}
\end{remark}

\section[The Dirichlet resolvent kernel for pcf fractals]{The Dirichlet resolvent kernel for p.c.f. self-similar fractals}
\label{sec:Dirichlet-resolvent-kernel}

In this section, we proceed through a sequence of lemmas which will allow us to prove Theorem~\ref{thm:resolvent-formula-preview}, which is stated in full in Theorem~\ref{thm:Dirichlet-resolvent-formula}. On a first reading, the reader may wish to read Theorem~\ref{thm:Dirichlet-resolvent-formula} first, and then work through the lemmas in reverse order. We take one hypothesis of Theorem~\ref{thm:resolvent-formula-preview} as a blanket assumption throughout this section:

\begin{assumption}\label{hyp:lambda-not-a-Dirichlet-eiv}
 None of the numbers $\gl_\word=\gm_\word r_\word\gl$, for $\word \in W_{\ast}$, is a Dirichlet eigenvalue of the Laplacian.
\end{assumption}

We construct the resolvent kernel formula according to the following rough outline:
\begin{enumerate}[(1)]
  \item We build a solution \fundeif{p} to the eigenfunction equation which takes the value 1 at one boundary point of $X$ and is 0 on the other boundary points.
  \item We show how $\gy\parlaq{p}$ may be written in terms of rescaled copies of \fundeif{p}, i.e., we decompose the solution around a point $p \in V_1 \less V_0$ into solutions for each cell containing $p$.
  \item We use this construction to obtain a solution on the cells of level $m$.
  \item We show how the $(m+1)$-level solution contains Dirac masses on $V_m$ which cancel with the Dirac masses of the $m$-level solution, so that the sum over $m$ is telescoping and yields a global solution.
\end{enumerate}
The first two steps are carried out in \S\ref{sec:The-basic-building-blocks-of-the-resolvent-kernel}. In \S\ref{sec:The-matrix-B}, we collect some properties of $B\parlaq{pq} := \sum_{K_j \ni q} \dn[K_j] \gy\parlaq{p}(q)$, as introduced in \eqref{eqn:basic-resolvent-solution-preview}\footnote{Recall that $\gy\parlaq{p}$ is the solution to the resolvent equation on level 1 as defined in \eqref{eqn:Bobmatrix-preview}.}. For each $\gl$, we think of $B\parlaq{pq}$ as the entries of a matrix in $p$ and $q$. Under Assumption~\ref{hyp:lambda-not-a-Dirichlet-eiv}, we show $B\parlam$ is symmetric, invertible, and that $\lim_{\gl \to 0} B\parlam = B\parlam[0]$.
Finally, the remaining two steps are carried out in \S\ref{sec:Construction-of-the-resolvent-kernel}.

Throughout this section, we will need to analyze the properties of a continuous function that satisfies the $\gl$-eigenfunction equation on all $1$-cells, but whose the Laplacian may fail to be in $L^2$.  Our motivation is that the Laplacian of such a function has Dirac masses at points $p \in V_1 \setminus V_0$ with coefficients that can be computed from the normal derivatives.  The following result is standard; see \cite[\S2.5]{Str06}, for example.

\begin{prop}\label{thm:Dirac-masses-for-laplacian}
  If $u$ is continuous and $\Lap u= v_j$ on each $1$-cell $K_j = F_j(X)$, then $\Lap u$ exists as a measure, as in \eqref{eqn:def:measure-Laplacian}, and
  \linenopar
  \begin{equation}\label{eqn:Dirac-masses-for-laplacian}
    \Lap u(x) = \sum_{j=1}^J v_j \charfn{K_j}(x)
      - \sum_{q \in V_1 \setminus V_0} \gd_{q}(x) \sum_{K_j \ni q} \dn[K_j] u(q)
  \end{equation}
  where $\gd_{q}(x)$ is a Dirac mass and $\dn[K_j] u(q)$ is the normal derivative of $u$ at $q$ with respect to the cell $K_j$, and the sum is expressed in the notation of Remark~\ref{rem:sum-over-sets-notation}.
\end{prop}

\subsection{The basic building blocks of the resolvent kernel}
\label{sec:The-basic-building-blocks-of-the-resolvent-kernel}

\begin{lemma}\label{thm:existence-of-fund-eif}
  For any $\gl$ that is not a Dirichlet eigenvalue of the Laplacian, and for each $p \in V_0$, there is a function $\fundeif{p}(x) \in \dom_{\sM} \Lap$, as in \eqref{eqn:def:measure-Laplacian}, which solves
  \linenopar
  \begin{equation}\label{eqn:fund-eif-problem}
    \begin{cases}
      (\gl\id - \Lap) \fundeif{p}(x) = 0, &\text{on } X, \\
      \fundeif{p}(q) = \gd_{pq}, & \forall q \in V_0,
    \end{cases}
  \end{equation}
  where $\gd_{pq}$ is the Kronecker delta.  Moreover, if $\gz_p$ is the harmonic function on $X$ with $\gz_p(q)=\gd_{pq}$, then
  \linenopar
  \begin{align}
    \fundeif{p} &= \gz_p - \gl \gq\parlaq{p}
      &&\text{on all of $X$, and} \label{eqn:fund-eif-problem-prop1} \\
    \dn \fundeif{p}(q)
    &= \dn \gz_p(q) - \gl \gk\parlaq{pq}
      &&\text{for } q \in V_0,
      \label{eqn:fund-eif-problem-prop2}
  \end{align}
  where $\gq\parlaq{p}$ and $\gk\parlaq{pq}$ are meromorphic functions of $\gl$ with poles at the Dirichlet eigenvalues of the Laplacian, and $\gk\parlaq{pq} = \gk\parlaq{qp}$.
\end{lemma}
  \begin{proof}
    Let $\{f_n\}$ denote the Dirichlet eigenfunctions of the Laplacian, with the corresponding eigenvalues $\gl_n$ arranged so that $\gl_{n+1}\geq\gl_n$; equality occurs iff $\gl_n$ has multiplicity greater than one. The functions $f_n$ may be assumed orthonormal, and their span is dense in $L^2$.  Consequently we may write $\gz_p = \sum_n a_p(n) f_n$. The function
    \linenopar
    \begin{equation*}
        \gq\parlaq{p}= \sum_n \frac{a_p(n)}{\gl-\gl_n} f_n
        \end{equation*}
    then satisfies
    \linenopar
    \begin{equation*}
        (\gl\id-\Lap) \gq\parlaq{p}
        = \sum_n \frac{a_p(n)}{\gl-\gl_n} (\gl-\gl_n) f_n
        =\sum_n a_p(n)f_n
        = \gz_p
        \end{equation*}
    in the $L^2$ sense.  We also see that $\Lap \gq\parlaq{p}= \sum_n \gl_n
    a_p(n) f_n/(\gl-\gl_n)$
    is $L^2$ convergent, so $\gq\parlaq{p}\in \dom_{L^2}\Lap$, as in \eqref{def:Laplacian}. In particular, $\gq\parlaq{p}$ is continuous and equal to zero on $V_0$.

    Define $\fundeif{p} := \gz_p - \gl\gq\parlaq{p}$. Then $(\gl\id-\Lap) \fundeif{p} = (\gl\id-\Lap) \gz_p - \gl \gz_p = 0$, and for $q \in V_0$,
    \linenopar
    \begin{equation}\label{eqn:fundeif=zeta=Dirac}
      \fundeif{p}(q)=\gz_p(q)=\gd_{pq}.
    \end{equation}

    To verify \eqref{eqn:fund-eif-problem-prop2}, we will need the fact that
    \begin{equation}\label{eqn:fundeif-normderivs-equal}
      \dn \fundeif{p}(q) = \dn\fundeif{q}(p),
    \end{equation}
    which follows by computing the normal derivatives as follows:
    \linenopar
    \begin{align*}
      \dn \fundeif{q}(p)- \dn \fundeif{p}(q)
      &= \sum_{s \in V_0} \left( \fundeif{p}(s) \dn \fundeif{q}(s)
        - \fundeif{q}(s) \dn \fundeif{p}(s) \right)
        &&\text{by \eqref{eqn:fundeif=zeta=Dirac}} \\
      &= \int_{X} \left( \fundeif{p}(x)\Lap\fundeif{q}(x)
        - \fundeif{q}(x)\Lap\fundeif{p}(x) \right)\, d\gm(x)
        &&\text{Gauss-Green} \\
      &= 0,\vstr[2.5] &&\Lap \fundeif{s} = \gl \fundeif{s}.
    \end{align*}
    Now \eqref{eqn:fund-eif-problem-prop2} follows via
    \linenopar
    \begin{align*}
      \dn \gz_p(q) - \dn \fundeif{p}(q)
      &= \dn \gz_p(q) - \dn \fundeif{q}(p)
        &&\text{by \eqref{eqn:fundeif-normderivs-equal}} \\
      &= \sum_{s\in V_0} \left( \fundeif{q}(s) \dn \gz_p(s)
        -  \gz_p(s) \dn \fundeif{q}(s)\right) \vstr[4]
        &&\text{by \eqref{eqn:fundeif=zeta=Dirac}} \\
      &=\int_X \left( \vstr[2.2] \gh\parlaq{q}(x)\Lap\gz_p(x)
        -  \gz_p(x)\Lap\fundeif{q}(x)\right)\, d\gm(x)
        &&\text{Gauss-Green} \\
      &= - \gl \int_X \gz_p(x)\fundeif{q}(x) \, d\gm(x)
        && \Lap\gz_p=0 \\
      &= \gl \int_X \left(\gl \gz_p(x) \theta\parlaq{q}(x)
        - \gz_p(x)\gz_q(x) \right)\,d\gm(x)
        &&\text{by \eqref{eqn:fund-eif-problem-prop1}} \\
      &= \gl^{2} \sum_n \frac{a_p(n) a_q(n)}{\gl-\gl_n}
        - \gl \sum_n a_p(n) a_q(n)\\
      &= \gl \sum_n \frac{\gl_n a_p(n) a_q(n)}{\gl-\gl_n}.
    \end{align*}
    Define for each $p,q \in V_0$ the functions
    \linenopar
    \begin{equation}\label{eqn:def-of-kappa}
      \gk\parlaq{pq} := \sum_n \frac{\gl_n a_p(n) a_q(n)}{\gl-\gl_n}
    \end{equation}
    so that $\dn \gz_q(p) - \dn \fundeif{q}(p) = \gl\gk\parlaq{pq}$.
    It is evident  that $\gk\parlaq{pq}$ is symmetric in $p$ and $q$.  It is also meromorphic in $\gl$ with poles at the points $\gl_n$, as may be verified by writing the expansion on the disc of radius $r$ centered at $z$ (where $z \neq \gl_n$ for any $n$, and $r = \inf_n|z-\gl_n|/2$) as follows:
    \linenopar
    \begin{equation*}
      \gk\parlaq{pq}
      = \sum_n \frac{\gl_n a_p(n) a_q(n)}{z-\gl_n}
        \sum_{k=0}^{\iy} \left(\frac{z-\gl}{z-\gl_n} \right)^k
      = - \sum_{k=0}^{\iy} (\gl-z)^k
        \sum_n \frac{\gl_n a_p(n) a_q(n)}{(\gl_n-z)^{k+1}}
    \end{equation*}
    and using the fact that $\{a_p(n)\}$ and $\{a_q(n)\}$ are in $\ell^{2}$ hence their product is in $\ell^{1}$, while $\gl_n/(\gl_n-z)^{k+1}$ is bounded for each $k$. Note that $r>0$ because the eigenvalues of a p.c.f. fractal have no finite accumulation point; cf. \cite[\S4.1]{Kig01}. An almost identical argument shows that $\gq\parlaq{p}$ is meromorphic in $\gl$ with values in $\dom_{L^2}\Lap$, so the proof is complete.
  \end{proof}

\begin{cor}\label{thm:relation-of-fundeif-to-psi}
  Let $p \in V_1 \less V_0$. If $r_j \gm_j \gl$ is not a Dirichlet eigenvalue for any $j$ with $p \in K_j$, then \gy and \gh are related via
  \linenopar
  \begin{equation}
    \gy\parlaq{p}(x) =
    \begin{cases}
      \gh\parlaq[r_j \gm_j \gl]{F_j^{-1} p} (F_j^{-1}x)
        &\text{if } p,x \in K_j, \\
      0 &\text{otherwise.}
    \end{cases}
    \label{eqn:relation-of-fundeif-to-psi}
  \end{equation}
\end{cor}
  \begin{proof}
    From \eqref{eqn:Laplacian-scaling} we have $\Lap (u \comp F_j^{-1}) = (r_j \gm_j)^{-1}(\Lap u) \comp F_j^{-1}$, for any $u$. Then from \eqref{eqn:basic-resolvent-solution-preview} and \eqref{eqn:fund-eif-problem}, one can observe that
    \linenopar
    \begin{align*}
      \begin{cases}
        (\gl\id-\Lap) \fundeif[r_j \gm_j \gl]{p} \comp F_j^{-1}=0, &\text{on } K_j = F_j(X) \\
        \fundeif[r_j \gm_j \gl]{p} \comp F_j^{-1}(q)=\gd_{F_j(p)q} &\forall q \in F_j(V_0).
      \end{cases}
        &\qedhere
    \end{align*}
  \end{proof}

\begin{remark}\label{rem:relation-of-fundeif-to-psi:compare-to-unit-interval}
  It is helpful to compare \eqref{eqn:relation-of-fundeif-to-psi} to the discussion of the unit interval, where \eqref{eqn:resolvent-interval-x,y<1/2} may be rewritten as
  \linenopar
  \begin{align*}
    R(x,y) =
    \begin{cases}
      \tfrac12 G\parlam[\gl/4](2x,2y) &\text{if } x,y \in K_j, \\
      0 &\text{otherwise.}
    \end{cases}
  \end{align*}
\end{remark}

\subsection{The matrix $B\parlam$}
\label{sec:The-matrix-B}
In the construction of the resolvent kernel, the matrix $B\parlam$ plays the same role as the transition matrix for the discrete Laplacian on $V_1$ in the corresponding argument of Kigami for the construction of the Dirichlet Green's function. We now collect some important properties of $B\parlam$ for use below.

\begin{lemma}\label{thm:Bpq-is-symmetric-and-contin-as-gl-to-0}
  The matrix $B\parlam$ is symmetric for any \gl, and $\lim_{\gl \to 0} B\parlam = B\parlam[0]$.
  \begin{proof}
    From \eqref{eqn:fundeif-normderivs-equal} we have $\dn[K_j] \gy\parlaq{p}(q) = \dn[K_j] \gy\parlaq{q}(p)$, and thus $B\parlaq{pq} = B\parlaq{qp}$.
    Then from \eqref{eqn:relation-of-fundeif-to-psi}, if $j_1, \dots j_k$ are those $j$ for which $K_j$ contains both $p$ and $q$, then
    \linenopar
    \begin{align}
      B\parlaq{pq}
      &= \sum_{i=1}^{k} \dn[K_{j_i}]  \left( \gh\parlaq[r_{j_i}\gm_{j_i} \gl]{F_{j_i}^{-1}(p)} \comp F_{j_i}^{-1} \right) (q) \notag \\
      &= \sum_{i=1}^{k} r_{j_i}^{-1} \dn[K_{j_i}]  \gh\parlaq[r_{j_i}\gm_{j_i}\gl] {F_{j_i}^{-1}(p)} \left( F_{j_i}^{-1}(q) \right) \notag \\
      &= \sum_{i=1}^{k} r_{j_i}^{-1} \dn[K_{j_i}] \gz_{F_{j_i}^{-1}(p)} \left(
        F_{j_i}^{-1}(q) \right)
        + \sum_{i=1}^{k} r_{j_i}^{-1} r_{j_i} \gm_{j_i} \gl \, \gk\parlaq[r_{j_i} \gm_{j_i} \gl]{F_{j_i}^{-1}(p) F_{j_i}^{-1}(q)} \notag
        &&\text{by \eqref{eqn:fund-eif-problem-prop2}} \\
      &= B\parlaq[0]{pq} + \gl \sum_{i=1}^{k} \gm_{j_i}
        \gk\parlaq[r_{j_i}\gm_{j_i}\gl]{F_{j_i}^{-1}(p)F_{j_i}^{-1}(q)}
        \label{eqn:Bpq-as-perturbation-of-B0}
    \end{align}
    in which the final sum term is a meromorphic function of $\gl$ with poles at those $\gl$ for which $r_{j_i}\gm_{j_i}\gl$ is a Dirichlet eigenvalue.  We used the  observation that the harmonic case with functions $\gz$ is just the case $\gl=0$. From \eqref{eqn:Bpq-as-perturbation-of-B0} it is also clear that $B\parlaq{pq} \to B\parlaq[0]{pq}$ as $\gl \to 0$.
  \end{proof}
\end{lemma}

As noted in the discussion following the statement of Theorem~\ref{thm:resolvent-formula-preview}, it is important that the action of $B\parlam$ on the subspace $V_1 \setminus V_0$ is invertible.

\begin{lemma}\label{thm:B-is-invertible}
  If $\gl$ is not a Dirichlet eigenvalue then $B\parlam$ is invertible.
\end{lemma}
\begin{proof}
  Suppose that $B\parlam = \left[B\parlaq{pq}\right]_{p,q \in V_1\setminus V_0}$ is not invertible, so there are values $a_{q}$ (not all 0) for which $\sum_{q\in V_1 \setminus V_0} B\parlaq{pq} a_{q}=0.$
  Define
  \linenopar
  \begin{equation*}
    u(x) := \sum_{q\in V_1 \setminus V_0} a_{q} \gy\parlaq{q}(x).
  \end{equation*}
  It is clear that $(\gl\id-\Lap)u=0$ on each $1$-cell, and that $u\evald{V_0}=0$. Now using the notation from Remark~\ref{rem:sum-over-sets-notation}, we compute the sum of the normal derivatives of $u$ over cells containing $p$, for any $p \in V_1 \setminus V_0$:\vstr[2.4]
  \linenopar
  \begin{align*}
    \sum_{K_j \ni p} \dn[K_j] u(p)
    &= \sum_{q\in V_1 \setminus V_0} a_{q} \sum_{K_j \ni p}
    \dn[K_j] \gy\parlaq{q}(p)\\
    &=\sum_{q\in V_1 \setminus V_0} a_{q} B\parlaq{qp} \\
    &=0,
  \end{align*}
  where the last equality follows by applying the symmetry established in Lemma~\ref{thm:Bpq-is-symmetric-and-contin-as-gl-to-0} to the initial assumption. So Proposition~\ref{thm:Dirac-masses-for-laplacian} implies $\Lap u$ is continuous. It follows that $(\gl\id-\Lap)u=0$ on $X$, so $u$ is a Dirichlet eigenfunction with eigenvalue $\gl$, which is a contradiction.
\end{proof}

The next result is used to prove Lemma~\ref{thm:Psi_m-has-dirac-masses} and also makes use of \eqref{eqn:relation-of-fundeif-to-psi}.

\begin{lemma}\label{thm:B-is-related-across-scales}
  For $p\in V_1 \setminus V_0$ and $q\in V_0$ we have
  \linenopar
  \begin{equation}\label{eqn:B-is-related-across-scales}
    \sum_{s\in V_1 \setminus V_0} B\parlaq{ps} \fundeif{q}(s)
    = - B\parlaq{pq}
  \end{equation}
\end{lemma}
\begin{proof}
  For a $1$-cell $K_j = F_j(X)$, the Gauss-Green formula gives
  \linenopar
  \begin{align*}
    &\sum_{s\in F_j(V_0)} \left(  \gy\parlaq{p}(s) \dn[K_j] \gh\parlaq{q}(s) - \fundeif{q}(s) \dn[K_j] \gy\parlaq{p}(s) \right) \\
    &\qq= \int_{K_j} \left( \gy\parlaq{p}(x) \Lap \fundeif{q}(x) - \gh\parlaq{q}(x) \Lap \gy\parlaq{p}(x)\right)\, d\gm(x) =0
  \end{align*}
  because both $\gy\parlaq{p}(x)$ and $\fundeif{q}(x)$ are Laplacian eigenfunctions with eigenvalue $\gl$ on each $1$-cell $K_j$. However for $s \in V_1$ we have $\gy\parlaq{p}(s) = \gd_{ps}$, so this becomes
  \linenopar
  \begin{align}\label{eqn:step1:thm:B-is-related-across-scales}
    \dn[K_j] \gh\parlaq{q}(p)
    = \sum_{s \in F_j(V_0)} \fundeif{q}(s) \dn[K_j] \gy\parlaq{p}(s).
  \end{align}
  The continuity of the Laplacian of $\fundeif{q}$ at $p \in V_1 \setminus V_0$ implies that its normal derivatives sum to zero, as indicated by Proposition~\ref{thm:Dirac-masses-for-laplacian}. Thus, summing over $1$-cells yields
  \linenopar
  \begin{align*}
    0
    = \sum_{j=1}^J \dn[K_j] \fundeif{q}(p)
    &= \sum_{j=1}^J \sum_{s \in F_j(V_0)} \fundeif{q}(s) \dn[K_j]
    \gy\parlaq{p}(s)
      &&\text{by \eqref{eqn:step1:thm:B-is-related-across-scales}} \\
    &= \sum_{s \in V_1} \fundeif{q}(s) \sum_{K_j \ni s} \dn[K_j] \gy\parlaq{p}(s)
      &&\text{interchange} \\
    &= \sum_{s \in V_0} \fundeif{q}(s) B\parlaq{ps}
      + \sum_{s \in V_1 \setminus V_0} \negsp[6] \fundeif{q}(s) B\parlaq{ps}
      &&\text{split} \\
    &= B\parlaq{pq}
      + \sum_{s \in V_1 \setminus V_0} \negsp[6] B\parlaq{ps}\fundeif{q}(s)
      &&\fundeif{q}(s)=\gd_{qs} \text{ on } V_0
  \end{align*}
  where we used the sum notation of Remark~\ref{rem:sum-over-sets-notation}.
\end{proof}

\subsection{Construction of the resolvent kernel}
\label{sec:Construction-of-the-resolvent-kernel}

Now that we have obtained some necessary properties of $B\parlam$, we can proceed with the development of a sequence of technical lemmas required for the proof of the main result. We begin with another corollary of Proposition~\ref{thm:Dirac-masses-for-laplacian}.

\begin{cor}\label{thm:Bpqsum-as-remainder}
  If $p \in V_1$ and \gl satisfies Assumption~\ref{hyp:lambda-not-a-Dirichlet-eiv}, then
  \begin{equation}\label{eqn:Bpqsum-as-remainder}
    (\gl\id-\Lap)\gy\parlaq{p} = \sum_{q \in V_1 \less V_0} B\parlaq{pq} \gd_{q}.
  \end{equation}
\end{cor}
  \begin{proof}
    With $\gy\parlaq{p}$ and $B\parlaq{pq}$ defined as in \eqref{eqn:basic-resolvent-solution-preview}--\eqref{eqn:Bobmatrix-preview}, this is clear from \eqref{eqn:Dirac-masses-for-laplacian}.
  \end{proof}

\begin{remark}\label{rem:resolvent-applied-to-gy}
  From the definition in \eqref{eqn:Bobmatrix-preview}, we have $B\parlaq{pq} = \sum_{K_j \ni q} \dn[K_j] \gy\parlaq{p}(q)$ for $q \in F_j(V_0)$. Thus Corollary~\ref{thm:Bpqsum-as-remainder} expresses the fact that an application of the resolvent to $\gy\parlaq{p}$ leaves behind nothing but a Dirac mass at every point of $V_1 \less V_0$, each weighted by the sum of the normal derivatives of $\gy\parlaq{p}$.
\end{remark}

The conclusion of the following lemma appears very technical but it expresses a straightforward idea: at each stage $m$, our formula for the resolvent corrects Dirac masses at the \nth[m] level and introduces new ones at the \nth[(m+1)]. Thus, summing over $m$ (as we do in Theorem~\ref{thm:Dirichlet-resolvent-formula}) produces a telescoping series. This makes precise the comment ``these are wiped away in the limit'' from the introductory discussion of the main result.

\begin{lemma}\label{thm:Psi_m-has-dirac-masses}
  Define $\gx\parlaq{p,m}$ to be the unique function solving
  \linenopar
  \begin{equation}\label{eqn:def:xi}
    \begin{cases}
      (\Lap-\gl)\gx\parlaq{p,m}=0, &\text{on all $m$-cells}, \\
      \gx\parlaq{p,m}(q) = \gd_{pq}, &\text{for $p \in V_m \less V_0$ and $q \in V_{m}$}.
    \end{cases}
  \end{equation}
  Then one has the identity
  \linenopar
  \begin{align*}
    \lefteqn{(\gl\id-\Lap_y) \negsp[5] \sum_{|\word|=m}  r_\word \gY\parlam[r_\word \gm_\word \gl] (F_\word^{-1}x,F_\word^{-1}y)} \hstr[18] \\
    &= \sum_{p\in V_{m+1}\setminus V_0} \negsp[12] \gx\parlaq{p,m+1}(x) \gd_{p}(y)
    - \sum_{q\in V_{m}\setminus V_0} \negsp[8] \gx\parlaq{q,m}(x) \gd_{q}(y).
  \end{align*}
\end{lemma}
\begin{proof}
  Since $\gY\parlam[r_\word \gm_\word \gl]$ is a sum of functions satisfying the \gl-eigenfunction equation on the level $1$ cells $K_j$, it is immediate that
  \linenopar
  \begin{equation*}
    (\gl-\Lap_y) \gY\parlam[r_\word \gm_\word \gl](F_\word^{-1}x,F_\word^{-1}y) = 0,
    \qq \text{for $y \notin V_{m+1}$.}
  \end{equation*}
  By Proposition~\ref{thm:Dirac-masses-for-laplacian}, we therefore need only compute the sum of normal derivatives at points of $V_{m+1}$.

  (1) First suppose that $z \in V_{m+1} \setminus V_{m}$ with $z=F_\word p$ for some $|\word|=m$ and $p \in V_1 \setminus V_0$, so that
  \linenopar
  \begin{equation*}
    \gY\parlam[r_\word \gm_\word \gl](F_\word^{-1}x,F_\word^{-1}z)
    = 
      \negsp[8] \sum_{s,t \in V_1 \setminus V_0} \negsp[8]
      G\parlaq[r_\word \gm_\word \gl]{st} \,
      \gy\parlaq[r_\word \gm_\word \gl]{s} \negsp[3](F_\word^{-1}x) \,
      \gy\parlaq[r_\word \gm_\word \gl]{t} \negsp[3](F_\word^{-1}z),
  \end{equation*}
  and collecting normal derivatives at $z$ yields
  \linenopar
  \begin{align}
    \lefteqn{\negsp[12] \sum_{F_\word(K_j)\ni z} \negsp[12] \dn[F_\word(K_j)] \negsp[5] \left(\gY\parlam[r_\word \gm_\word \gl](F_\word^{-1}x, \vstr[2.2]F_\word^{-1}z)\right)}
      \hstr[20] \notag \\
    &= \negsp[12] \sum_{s,t\in V_1\setminus V_0} \negsp[8]
      G\parlaq[r_\word \gm_\word \gl]{st}
      \gy\parlaq[r_\word \gm_\word \gl]{s} \negsp[3](F_\word^{-1}x)
      \sum_{F_\word(K_j)\ni z} \negsp[12] \dn[F_\word(K_j)] \gy\parlaq[r_\word \gm_\word \gl]{t} \negsp[3](F_\word^{-1}z) \notag \\
    &= r_\word^{-1} \negsp[12] \sum_{s,t\in V_1\setminus V_0} \negsp[8]
      G\parlaq[r_\word \gm_\word \gl]{st} \,
      \gy\parlaq[r_\word \gm_\word \gl]{s} \negsp[3](F_\word^{-1}x) \,
      B\parlaq[r_\word \gm_\word \gl]{tp},
      \label{eqn:Psi_m-has-dirac-masses:comp1}
  \end{align}
  because
  \linenopar
  \begin{align*}
    B\parlaq[r_\word \gm_\word \gl]{tp}
    &= \sum_{K_j \ni p} \dn[K_j] \gy\parlaq[r_\word \gm_\word \gl]{t}(p)
      &&\text{by \eqref{eqn:Bobmatrix-preview}} \notag \\
    &= \sum_{K_j \ni F_\word^{-1} z} \dn[K_j] \gy\parlaq[r_\word \gm_\word  \gl]{t}(F_\word^{-1} z)
      &&p = F_\word^{-1} z \in K_j \notag \\
    &= r_\word \sum_{F_\word (K_j) \ni z} \dn[F_\word (K_j)] \left(\gy\parlaq[r_\word \gm_\word \gl]{t} \comp F_\word^{-1}\right)(z),
  \end{align*}
  where the last line follows from $\dn[K_\word] u(F_\word^{-1} q_i) = r_\word \dn (u \comp F_\word^{-1})(q_i)$; cf. \eqref{eqn:def:normal-derivative}.

  Continuing the computation from \eqref{eqn:Psi_m-has-dirac-masses:comp1} and making use of $G := B^{-1}$, we have
  \linenopar
  \begin{align*}
    \sum_{F_\word(K_j)\ni z} \negsp[12] \dn[F_\word(K_j)] \negsp[5] \left(\gY\parlam[r_\word \gm_\word \gl](F_\word^{-1}x, \vstr[2.2]F_\word^{-1}z)\right)
    &= r_\word^{-1} \sum_{s\in V_1\setminus V_0} \negsp[8] \gd_{sp}
        \left( \gy\parlaq[r_\word \gm_\word \gl]{s} (F_\word^{-1}(x)) \right)  \\
    &= r_\word^{-1} \gy\parlaq[r_\word \gm_\word \gl]{p} (F_\word^{-1}(x)) \\
    &= r_\word^{-1} \gx\parlaq{z,m+1}(x)
  \end{align*}
  thus showing that $(\gl\id-\Lap_y)$ has a Dirac mass $\gx\parlaq{z,m+1}(x) \gd_{z}(y)$ at $z \in V_{m+1} \setminus V_{m}$.

  \pgap

  (2) Next consider a point $z \in V_{m} \setminus V_0$. In this case there are several words $\word_i$ for which $z = F_{\word_i}(p_i)$ for some $p_i \in V_0$.  For such a word \word and such a $p$ we substitute from Lemma~\ref{thm:B-is-related-across-scales} into \eqref{eqn:Psi_m-has-dirac-masses:comp1}, obtaining
  \linenopar
  \begin{align}
    \sum_{K_j \ni p} \negsp[3] \dn[F_\word(K_j)] \negsp[3] \left(r_\word
    \gY\parlam[r_\word \gm_\word \gl](F_\word^{-1}x, \vstr[2.2] F_\word^{-1}z)\right)
    &= -\negsp[10] \sum_{q,s,t \in V_1\setminus V_0} \negsp[12]
      G\parlaq[r_\word \gm_\word \gl]{s,t} \,
      \gy\parlaq[r_\word \gm_\word \gl]{s} \negsp[3](F_\word^{-1}x) \,
      B\parlaq[r_\word \gm_\word \gl]{tq} \fundeif[r_\word \gm_\word \gl]{p}(q) \notag\\
    &= -\negsp[8] \sum_{q,s \in V_1\setminus V_0} \negsp[10] \gd_{sq} \,
      \gy\parlaq[r_\word \gm_\word \gl]{s} \negsp[3](F_\word^{-1}x) \,
      \fundeif[r_\word \gm_\word \gl]{p}(q) \notag \\
    &= -\negsp[6] \sum_{q \in V_1\setminus V_0} \negsp[8]
      \gy\parlaq[r_\word \gm_\word \gl]{q} \negsp[3](F_\word^{-1}x) \,
      \fundeif[r_\word \gm_\word \gl]{p}(q).
      \label{eqn:Psi_m-has-dirac-masses:comp2}
  \end{align}
  The result is clearly a piecewise $\gl$-eigenfunction on level $(m+1)$ with respect to the $x$ variable, so is determined by its values on $V_{m+1}$.  In each of the terms~\eqref{eqn:Psi_m-has-dirac-masses:comp2}, the values are nonzero only at the points of $V_{m+1}$ that neighbor $z$ in $F_\word(X)$, and they are easily seen to coincide with $\gx\parlaq{z,m+1}-\gx\parlaq{z,m}$ at these points. Summing over all cells, we conclude that at each  $z \in V_{m} \setminus V_0$ the operator $(\gl\id-\Lap)$ has a Dirac mass $\left(\gx\parlaq{z,m+1} - \gx\parlaq{z,m}\right)\gd_{z}(y)$, and the result follows.
\end{proof}

\begin{theorem}\label{thm:Dirichlet-resolvent-formula}
  Let $\gy\parlaq{p}$ be the solution to the resolvent equation at level 1, i.e.
  \linenopar
  \begin{equation}\label{eqn:basic-resolvent}
    \begin{cases}
      (\gl \id - \Lap)\gy\parlaq{p} = 0, &\text{on each } K_j=F_j(X), \\
      \gy\parlaq{p}(q) = \gd_{pq}, &\text{for $p \in V_1 \less V_0$ and $q \in V_1$},
    \end{cases}
  \end{equation}
  where $\gd_{pq}$ is the Kronecker delta.

  Define the kernel
  \linenopar
  \begin{align}
    G\parlam(x,y)
    =& \sum_{\word \in W_\ast} r_\word \gY\parlam[r_\word \gm_\word \gl] (F_\word^{-1} x, F_\word^{-1}y),
      \label{eqn:resolvent-formula} \\
    \text{where }
    \gY\parlam(x,y)
    :=& \sum_{p,q \in V_1 \less V_0} \negsp[12] G\parlaq{pq} \gy\parlaq{p}(x) \gy\parlaq{q}(y).
    \label{eqn:Psi}
  \end{align}
  The coefficients $G\parlaq{pq}$ in \eqref{eqn:Psi} are the entries of the inverse of the matrix $B$ given by
  \linenopar
  \begin{equation}\label{eqn:Bobmatrix}
    B\parlaq{pq} := \sum_{K_j \ni q} \dn[K_j] \gy\parlaq{p}(q),
    \qq q \in F_j(V_0),
  \end{equation}
  the sum taken over all 1-cells containing $q$.

  For $\gl$ satisfying Assumption~\ref{hyp:lambda-not-a-Dirichlet-eiv}, $G\parlam(x,y)$ is symmetric and continuous in $x$ and $y$, and is in $\dom_{\sM} \Lap_y$ with $(\gl\id-\Lap_y) G\parlam(x,y) = \gd_{x}(y)$.  As it vanishes on $V_0$, it is the Dirichlet resolvent of the Laplacian.
\end{theorem}
\begin{proof}
  The symmetry of $G\parlam(x,y)$ is obvious. Next, note that
  \linenopar
  \begin{align*}
    \lefteqn{(\gl\id-\Lap_y) \sum_{m=0}^{M} \sum_{|\word|=m}  r_\word \gY\parlam[r_\word \gm_\word \gl](F_\word^{-1}x,F_\word^{-1}y)}
    \hstr[20] \\
    &=\sum_{m=0}^{M} \left(\sum_{p \in V_{m+1} \less V_0} \negsp[10] \gx\parlaq{p,m+1}(x)\gd_{p}(y) - \sum_{p \in V_{m} \less V_0} \negsp[10] \gx\parlaq{p,m}(x)\gd_{p}(y)  \right) \\
    &= \negsp[8] \sum_{p\in V_{M+1} \setminus V_0} \negsp[8] \gx\parlaq{p,M+1}(x)\gd_{p}(y)
  \end{align*}
  by Lemma~\ref{thm:Psi_m-has-dirac-masses}, so that
  \linenopar
  \begin{align*}
    \lim_{M \to \iy} (\gl\id-\Lap_y) \sum_{m=0}^{M} \sum_{|\word|=m}  r_\word \gY\parlam[r_\word \gm_\word \gl](F_\word^{-1}x,F_\word^{-1}y) = \gd_{x}(y),
  \end{align*}
  in the sense of weak-$\ast$ convergence. It follows that $G\parlam(x,y)$ is in $\dom_{\sM}(\Lap_y)$ and that $(\gl\id-\Lap_y) G\parlam(x,y) = \gd_{x}(y)$.

  All that remains is to see that $G\parlam(x,y)$ is continuous. However, Lemma~\ref{thm:Bpq-is-symmetric-and-contin-as-gl-to-0} shows $B\parlaq[r_\word \gm_\word \gl]{pq} \to B\parlaq[0]{pq}$ as $|\word|\to\iy$, and hence $G\parlaq[r_\word \gm_\word \gl]{pq} \to G\parlaq[0]{pq}$. In a similar manner, the relation $\gh\parlaq[r_\word \gm_\word \gl]{p} = \gz_p + r_\word \gm_\word \gl \gq\parlaq[r_\word \gm_\word \gl]{p}$ from Lemma~\ref{thm:existence-of-fund-eif} shows that $\gh\parlaq[r_\word \gm_\word \gl]{p} \to \gz_p$ as $|\word|\to\iy$; in particular we find that $\gy\parlaq[r_\word \gm_\word \gl]{p} \to \gy\parlaq[0]{p}$, and the latter is piecewise harmonic and bounded by $1$.  The conclusion is that  $\gY\parlam[r_\word \gm_\word \gl]$ is bounded as $|\word|\to\iy$, and since $r_\word$ is a product of $|\word|$ terms, all of which are bounded by $\max_{i} r_{i}<1$,
  \linenopar
  \begin{equation*}
      G\parlam(x,y) = \sum_{m=0}^{\iy} \sum_{|\word|=m} r_\word \gY\parlam[r_\word \gm_\word \gl](F_\word^{-1}x,F_\word^{-1}y)
      \end{equation*}
  is bounded by a convergent geometric series. Note that, for each $m$, only a finite number of terms in the second sum are nonzero. As all terms are continuous, so is $G\parlam$.
\end{proof}

\section[The Neumann resolvent kernel for pcf fractals]{The Neumann resolvent kernel for p.c.f. self-similar fractals}
\label{sec:Neumann-resolvent-kernel}

In Theorem~\ref{thm:Neumann-resolvent-formula}, we give the formula for the Neumann resolvent kernel.

\begin{lemma}\label{thm:the-Luke-matrix}
  If $\gl$ is not a Neumann eigenvalue then there is $C\parlaq{pq}$ such that
  \linenopar
  \begin{equation*}
    \sum_{q \in V_0} C\parlaq{pq} \nd \gh\parlaq{q}(x) = \gd_{px}
  \end{equation*}
  for $x \in V_0$, and $C\parlaq{pq}$ is symmetric in $p$ and $q$.
\end{lemma}
\begin{proof}
  Since $\gl$ is not a Neumann eigenvalue, the set of vectors $\left\{\left(\nd \gh\parlaq{q}(x)\right)_{x \in  V_0}\right\}_{q\in V_0}$ is linearly independent, whence the existence of the $C\parlaq{pq}$ is immediate.  Symmetry follows from \eqref{eqn:fundeif-normderivs-equal} because the matrix $\left[C\parlaq{pq}\right]$ is the inverse of the symmetric matrix $\left[\nd \gh\parlaq{p}(q)\right]$.
\end{proof}

From this and Theorem~\ref{thm:Dirichlet-resolvent-formula} we may readily deduce the following result.

\begin{theorem}\label{thm:Neumann-resolvent-formula}
  If $\gl$ satisfies Assumption~\ref{hyp:lambda-not-a-Dirichlet-eiv} and also is not a Neumann eigenvalue, then
  \linenopar
  \begin{equation}\label{eqn:Neumann-resolvent-formula}
    G\parlaq{N}(x,y)
    = G\parlam(x,y) + \sum_{p,q\in V_0} C\parlaq{pq} \gh\parlaq{p}(x) \gh\parlaq{q}(y)
  \end{equation}
  is symmetric, is in $\dom_{\sM}(\Lap_y)$, and satisfies $(\gl-\Lap_y) G\parlaq{N}(x,y) = \gd_x(y)$ on $X \less V_0$.  It has vanishing normal derivatives on $V_0$ and is therefore the Neumann resolvent kernel of the Laplacian.
\end{theorem}
\begin{proof}
  The symmetry of $G\parlaq{N}(x,y)$ is immediate from the symmetry of $G\parlam(x,y)$ and of $C\parlaq{pq}$.  Both $G\parlam(x,y)$ and $\gh\parlaq{p}(y)$ are in $\dom_{\sM}(\Lap_{y})$ and $(\gl-\Lap_{y}) \gh\parlaq{p}(y) =0$ on $X \less V_0$ so $(\gl-\Lap_{y}) G\parlaq{N}(x,y)=(\gl-\Lap_{y})G\parlam(x,y) = \gd_x(y)$ on $X\less V_0$.

  It remains to prove the assertion about the normal derivatives. 
  We will use the notation $\ndsec G\parlam$ for the normal derivative of $G\parlam(x,y)$ with respect to its second variable. Since $G\parlam(x,y) \in \dom_{\sM}(\Lap_y)$ it has a normal derivative at $p \in V_0$, and by the Gauss-Green formula,
  \linenopar
  \begin{align}\label{eqn:normal-at-p-of-G}
    \ndsec G\parlam(x,p)
    &= \sum_{s\in V_0} \left( \ndsec G\parlam(x,s) \gh\parlaq{p}(s) - G\parlam(x,s) \partial_n \gh\parlaq{p}(s) \right) \notag \\
    &= \int_{X} \left( \left(\Lap_{s} G\parlam(x,s)\right) \gh\parlaq{p}(s) - G\parlam(x,s)
    \left(\Lap_{s}\gh\parlaq{p}(s) \right)\right) \, d\gm(s) \notag \\
    &= \int_{X} \left(\Lap_{s} - \gl\right) G\parlam(x,s) \gh\parlaq{p}(s) \, d\gm(s) \notag \\
    &= -\gh\parlaq{p}(x)
  \end{align}
  where at the first step we used that $G\parlam(x,s)=0$ for $s \in V_0$ and at the last step we used $\left(\Lap_{s}-\gl\right) G\parlam(x,s)=-\delta_{x}(s)$ as a measure.
  It follows that at each $p \in V_0$, the normal derivative of \eqref{eqn:Neumann-resolvent-formula} vanishes:
  \linenopar
  \begin{align*}
    \ndsec G\parlaq{N}(x,p)
    &= -\gh\parlaq{p}(x) + \ndsec \negsp[6] \sum_{q,s\in V_0} C\parlaq{qs} \gh\parlaq{q}(x) \gh\parlaq{s}(p)
      &&\text{by \eqref{eqn:normal-at-p-of-G}} \\
    &= -\gh\parlaq{p}(x) + \sum_{q\in V_0} \gd_{qx} \gh\parlaq{q}(x)
      &&\text{by Lemma~\ref{thm:the-Luke-matrix}} \\
    &= 0. &&\qedhere
  \end{align*}
\end{proof}

\section{Example: the Sierpinski gasket $S\negsp[2]G$}
\label{sec:example-sg}

Recall the harmonic extension algorithm as described in \cite[\S1.3]{Str06}: if the values of a function $u$ are specified at the points of $V_0$ and written as a vector
\linenopar
\begin{align*}
  u \evald{V_0}
  = \left[\begin{array}{c} u(p_0) \\ u(p_1) \\ u(p_2) \end{array}\right],
\end{align*}
then the harmonic extension of $u$ to $F_i(V_0)$ (the boundary points of the 1-cell $F_i(SG)$) is given by
\linenopar
\begin{align*}
  u \evald{F_i V_0}
  = A_i u \evald{V_0}
  = \left[\begin{array}{c} u(F_i p_0) \\ u(F_i p_1) \\ u(F_i p_2) \end{array}\right],
\end{align*}
where
\linenopar
\begin{align*}
  A_0 = \frac15
    \left[\begin{array}{ccc}
      5 & 0 & 0 \\
      2 & 2 & 1 \\
      2 & 1 & 2 \\
    \end{array}\right], \q
  A_1 = \frac15
    \left[\begin{array}{ccc}
      2 & 2 & 1 \\
      0 & 5 & 0 \\
      1 & 2 & 2 \\
    \end{array}\right], \q \text{and} \q
  A_2 = \frac15
    \left[\begin{array}{ccc}
      2 & 1 & 2 \\
      1 & 2 & 2 \\
      0 & 0 & 5 \\
    \end{array}\right]
\end{align*}
are the harmonic extension matrices. In general, $u \evald{F_\word V_0} = A_\word u \evald{V_0}$, where $A_\word = A_{\word_m} \cdots  A_{\word_1}$. Thus, the harmonic extension matrices allow one to construct a harmonic function with specified boundary values. Similarly, spectral decimation provides matrices which allow one to construct an eigenfunction with specified boundary values. For example,
\linenopar
\begin{align}\label{eqn:eigenfunction-extension-matrix}
  A_{0}(\gl)
  = \frac{1}{(5-\gl)(2-\gl)}
    \left[\begin{array}{ccc}
      (5-\gl)(2-\gl) & 0 & 0 \\
      (4-\gl) & (4-\gl) & 2 \\
      (4-\gl) & 2 & (4-\gl)
    \end{array}\right]
\end{align}
is the analogue of $A_0 = A_0(0)$. By the usual caveats of spectral decimation, these extension matrices can only be used when \gl is not a (Dirichlet) eigenvalue.

\begin{remark}[Spectral decimation]\label{rem:spectral-decimation}
  A very brief outline of the method of spectral decimation is as follows.
  \begin{enumerate}[(1)]
    \item Begin on some level $m=m_0$ with $u_m$ and $\gl_m$ that satisfy $-\Lap_m u_m = \gl_m u_m$ on $V_m \less V_0$.
    \item Extend $u_m$ to a function $u_m$ on $V_{m+1} \less V_0$ by comparing the eigenvalue equations from each level.
    \item Obtain a collection of extension matrices, one for each mapping in the original IFS, and a rational function $\varrho$ which relates the eigenvalues on one level to the eigenvalues on the previous level by $\varrho(\gl_m) = \gl_{m-1}$.
    \item Inductively construct a sequence $\{\gl_{m_o}, \gl_{m_o+1}, \gl_{m_o+2}, \dots\}$ by choosing $\gl_{m+1}$ from the set $\varrho^{-1}(\gl_{m})$ for each $m$.
  \end{enumerate}
  For every such sequence that converges, $\ga \lim_{m \to \iy} \gb^m \gl_m$ will be an eigenvalue of \Lap on $X$, where \ga and \gb are constants specific to $X$. For the Sierpinski Gasket $S\negsp[2]G$, $\ga = \frac32$ and $\gb = 3 \frac53 = 5$. Note that the calculations in (2)--(3) will forbid certain choices, so some care must be taken in the construction of $\{\gl_m\}$. See \cite[\S3]{Str06} for more details.
\end{remark}

To obtain the numbers $B\parlaq{pq}$ (appearing in \eqref{eqn:Bobmatrix-preview}) for the Sierpinski Gasket $S\negsp[2]G$, we find the normal derivatives of the eigenfunction that has boundary values $(1,0,0)$, as computed at each point of $V_0$. If $(\gl\id-\Lap)u=0$ but $u$ is not a Dirichlet eigenfunction, then consider $u$ on $F_{0}^{m}(V_{0})$. By spectral decimation, this is given by
\linenopar
\begin{equation*}
  u\evald{F_{0}^{m}(V_{0})}
  = A_{0}(\gl_m) \dotsm A_{0}(\gl_{1}) u\evald{V_{0}}
\end{equation*}
where the matrix $A_0(\gl)$ is as in \eqref{eqn:eigenfunction-extension-matrix}. We actually only need the values of the \emph{normal derivative}
\linenopar
\begin{equation}\label{eqn:def:normal-derivative-SG}
  \dn u(q_i)
  = \lim_{m \to \iy} \left( \frac53 \right)^{m} 
    \left(2u(q_i) - u(F_i^m q_{i-1}) - u(F_i^m q_{i+1}) \right),
    \qq q_i \in V_0.
\end{equation}
The factor $\frac53$ arises here because $r_j = \frac35$ for each $j$ in \eqref{eqn:def:energy-form} for $S\negsp[2]G$; see also \eqref{eqn:def:normal-derivative}. The calculation of $r_j = \frac35$ is given in \cite[\S1.3]{Str06}.

It is extremely easy to compute the normal derivatives of a harmonic function: one does not need to compute the limit, as all terms of the sequence are equal; see \cite[(2.3.9)]{Str06}. Therefore, our approach is to obtain a harmonic function which coincides with $u$ on $F_0^m(V_0)$. The limit of the normal derivatives of these harmonic functions will be the normal derivative of $u$. An alternative interpretation would be to interpret the harmonic functions on $S\negsp[2]G$ as the analogue of the linear functions on $I$. Consequently, the tangent to a point of $S\negsp[2]G$ should be given by a harmonic function plus a constant, provided the tangent exists. This is the motivating idea of \cite{DRS07}.

Multiplication by $A_0^{-m}$ allows one to find the required harmonic function at stage $m$; rewriting the normal derivative \eqref{eqn:def:normal-derivative-SG} in vector notation, one has
\linenopar
\begin{align*}
  \left( \frac53 \right)^{m} (2,-1,-1) \cdot u\evald{F_{0}^{m}(V_{0})}
  &= (2,-1,-1) \cdot A_{0}^{-m} u\evald{F_{0}^{m}(V_{0})} \\
  &= (2,-1,-1) \cdot A_{0}^{-m} A_{0}(\gl_m) \dotsm A_{0}(\gl_{1}) u\evald{V_{0}}.
\end{align*}
It therefore suffices to understand the limit $\lim_{m}  A_{0}^{-m} A_{0}(\gl_m) \dotsm A_{0}(\gl_{1})$; this was computed in \cite{DRS07}. The following theorem is the main result of \cite{DRS07}, taken with $m_0=0$.

\begin{theorem}\label{thm:tangents-explicitformula}
Let $\ga=(0,1,1)^{T}$, $\beta=(0,1,-1)^{T}$, $\gg_{m}=(4,4-\gl_m,4-\gl_m)^{T}$.
If neither of the values $2$ or $5$ occur in the sequence $\gl_m$, then
\linenopar
\begin{gather*}
    \lim_{k\rightarrow\iy}
        A_{0}^{-k} \cdot A_{0}(\gl_{0+k})\dotsm A_{0}(\gl_{0+1})\, \ga
    = \frac{ 4\gl } {3\cdot5^{0}\gl_{0}(2-\gl_{0+1})}
        \prod_{j=2}^{\iy} \left( 1 -\frac{\gl_{0+j}}{3} \right)\, \ga \\
    \lim_{k\rightarrow\iy}
        A_{0}^{-k} \cdot A_{0}(\gl_{0+k})\dotsm A_{0}(\gl_{0+1})\, \beta
    = \frac{2\gl}{3\cdot5^{0}\gl_{0}}\, \beta\\
    \lim_{k\rightarrow\iy}
        A_{0}^{-k} \cdot A_{0}(\gl_{0+k})\dotsm A_{0}(\gl_{0+1})\, \gg_{0}
    = (4,4,4)^{T}
    \end{gather*}
\end{theorem}

In particular, this can be used to get the desired normal derivative.  We know that all we need do is compute
\linenopar
\begin{equation}\label{eqn:normal-derivative-vectorform}
  (2,-1,-1) \cdot \left( \lim_{m} A_{0}^{-m} A_{0}(\gl_m) \dotsm A_{0}(\gl_{1}) \right)
  u\evald{V_{0}}.
\end{equation}
The boundary data $u\evald{V_{0}}$ is be taken to be $(1,0,0)$ when computing the normal derivative at the point $p$ where $u(p)=1$, and $(0,1,0)$ at a point where $u(p)=0$ (these two points are the same by symmetry). \vstr[2.2] Writing
\linenopar
\begin{equation*}
  \left[\begin{array}{ccc} 1\\0\\0 \end{array}\right]
  = \frac{1}{4} \left[\begin{array}{ccc} 4\\4-\gl_{0}\\4-\gl_{0} \end{array}\right]
      - \frac{4-\gl_{0}}{4}\left[\begin{array}{ccc} 0\\1\\1\end{array}\right],
  \qq\text{and} \qq
  \left[\begin{array}{ccc} 0\\1\\0 \end{array}\right]
  =\frac{1}{2} \left[\begin{array}{ccc} 0\\1\\1\end{array}\right] +
      \frac{1}{2} \left[\begin{array}{ccc} 0\\1\\-1\end{array}\right],
\end{equation*}
we find that
\linenopar
\begin{align*}
  \lim_{m} A_{0}^{-m} & A_{0}(\gl_m) \dotsm A_{0}(\gl_{1})
    \left[\begin{array}{ccc} 1 \\ 0 \\ 0 \end{array}\right] \\
  &= \left[\begin{array}{ccc} 1 \\ 1 \\ 1 \end{array}\right]
    - \frac{4-\gl_{0}}{4} \frac{4\gl}{3\gl_0(2-\gl_1)} \prod_{j=2}^{\iy}
    \left(1 - \frac{\gl_j}{3}\right)
    \left[\begin{array}{ccc} 0 \\ 1 \\ 1 \end{array}\right]
\end{align*}
and by \eqref{eqn:normal-derivative-vectorform}, the normal derivative is
\linenopar
\begin{equation*}
  \frac{2(4-\gl_0) \gl} {3\gl_0(2-\gl_1)}
  \prod_{j=2}^{\iy} \left(1 -\frac{\gl_j}{3}\right).
\end{equation*}

The normal derivative at the other point is computed by first finding
\linenopar
\begin{align*}
  \lim_{m} A_{0}^{-m} A_{0}(\gl_m)&\dotsm A_{0}(\gl_{1}) \left[\begin{array}{ccc} 0\\1\\0 \end{array}\right] \\
  &= \frac{1}{2} \frac{ 4\gl } {3\gl_{0}(2-\gl_{1})}
    \prod_{j=2}^{\iy} \left( 1 -\frac{\gl_{j}}{3} \right) \left[\begin{array}{ccc} 0\\1\\1\end{array}\right]
    + \frac{1}{2} \frac{2\gl}{3\gl_{0}} \left[\begin{array}{ccc} 0\\1\\-1\end{array}\right]
\end{align*}
and then taking the inner product with $(2,-1,-1)$, which cancels the second vector to leave
\linenopar
\begin{equation*}
	\frac{ - 4\gl } {3\gl_{0}(2-\gl_{1})} \prod_{j=2}^{\iy} \left( 1 -\frac{\gl_{j}}{3} \right).
	\end{equation*}

It seems logical at this point to define a function
\linenopar
\begin{equation}\label{def:tau}
	\gt(\gl) = \frac{4\gl } {3\gl_{0}(2-\gl_{1})} \prod_{j=2}^{\iy} \left( 1 -\frac{\gl_{j}}{3} \right)
	\end{equation}
and to write the normal derivative at the point where the $1$ occurs as $(4-\gl_{0})\gt(\gl)/2$ and that
at the point where the $0$ occurs as $-\gt(\gl)$.  We note that for a non-Dirichlet eigenfunction, none of the values $2,5,6$ occur in $\gl_m$ for $m \geq 1$ so the term $(2-\gl_{1})$ in the denominator cannot be zero.  It follows that $3$ does not occur for $m \geq 2$ and therefore that $\gt(\gl) \neq 0$ in this case.  
An exception to our formula as currently written occurs when $\gl=0$, because then also $\gl_{0}=0$, but the function $\gt(\gl)$ is easily shown to have a continuous extension to $\gl=0$ with $\gt(0)=1$; cf. \cite{DRS07}. With this correction, our formula is also valid for the harmonic case.

It is now easy to write the entries of the matrix $B\parlaq{pq}$ appearing in \eqref{eqn:Bobmatrix-preview}. The term $B\parlaq{pp}$ has two copies of the normal derivative $(4-\gl_{0})\gt(\gl)/2$, and the term $B\parlaq{pq}$ has a single copy of $-\gt(\gl)$ at each $q\in V_{1}$ that is not equal to $p$.  Both are on $1$-cells rather than the whole of $S\negsp[2]G$, so there is an extra factor $5/3$ in their normal derivatives.  As a result, the matrix is
\linenopar
\begin{equation*}
	B = \frac53\left[\begin{array}{ccc}
		(4-\gl_{0})\gt(\gl) & -\gt(\gl) & -\gt(\gl)\\
		-\gt(\gl) & (4-\gl_{0})\gt(\gl) & -\gt(\gl)\\
		-\gt(\gl)& -\gt(\gl) & (4-\gl_{0})\gt(\gl)
		\end{array}\right]
	\end{equation*}
and we should invert this to get the matrix $G_{pq}$ for the Green's function. Since
\linenopar
\begin{equation*}
	\det \left[\begin{array}{ccc}
		a & b & b \\
		b & a & b \\
		b & b & a
	\end{array}\right]
  = (a-b)^{2}(a+2b),
\end{equation*}
the matrix $B$ is invertible iff $\gl_0 \neq 2,5$, in which case
\linenopar
\begin{equation*}
	G\parlam
	= \frac{3}{5(5-\gl_{0})(2-\gl_{0}) \gt(\gl)}
		\left[\begin{array}{ccc}
		(3-\gl_{0}) & 1 & 1\\
		1& (3-\gl_{0})& 1\\
		1& 1& (3-\gl_{0})
		\end{array}\right].
	\end{equation*}
Note that this is consistent with the harmonic case where $\gl_0 = 0$ and $\gt(0)=1$ gives factors $9/50$ for $G\parlaq{pp}$ and $3/50$ for $G\parlaq{pq}$ with $p\neq q$; see \cite[(2.6.25)]{Str06}.

\section{Example: $S\negsp[2]G_3$, a variant of the Sierpinski gasket}
\label{sec:example-sg3}

\subsection{The Laplacian on $S\negsp[2]G_3$}

\begin{figure}
  \centering
  \scalebox{0.80}{\includegraphics{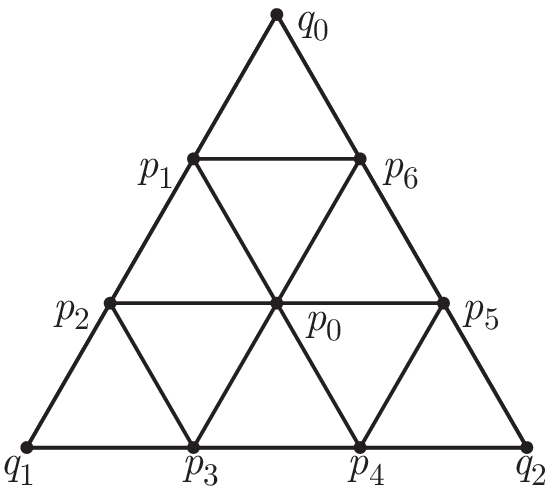}}
  \caption{The 1-cells of $S\negsp[2]G_3$.}
  \label{fig:SG3}
\end{figure}

The fractal $S\negsp[2]G_3$ is obtained from an IFS consisting of 6 contraction mappings, each with scaling ratio $\frac13$, as indicated in Figure~\ref{fig:SG3}.
The details of the spectral decimation method for $S\negsp[2]G_3$ have been worked out independently in \cite{BCDEHKMST07,DrSt07,Zhou09a,Zhou09b}.
Note that $p_0$ is contained in three 1-cells of $S\negsp[2]G_3$, in contrast to each of the other points $p_i$ of $V_1 \less V_0$, which are contained in two. For this reason, we define the graph Laplacian on $S\negsp[2]G_3$ as
\linenopar
\begin{equation}\label{eqn:def-Laplacian-Gamma-m}
  \Lap_m u(x) = \frac{1}{deg(x)} \sum_{y \nbr[m] x} (u(y)-u(x)),
\end{equation}
where $deg(x)$ is the number of $m$-cells containing $x$. From \cite[\S4.4]{Str06}, we have
\linenopar
\begin{equation}\label{eqn:def-relationship-Laplacian-SG3-Gamma-m}
  \Lap_\gm u(x)
  = \lim_{m\to \infty} r^{-m} \left( \int_K h_x^{(m)} \, d \gm \right)^{-1} deg(x) \Lap_m u(x).
\end{equation}
The renormalization constant $r$ will be computed in \S\ref{sec:Computation-of-the-normal-derivatives}.

Let $p_x$ denote a vertex where $u$ takes the value $x$ as
depicted in Figure \ref{fig:SG3-100}, then by
(\ref{eqn:def-Laplacian-Gamma-m}), the symmetric eigenvalue
equations on $V_m$ are
\linenopar
\begin{align*}
  \gD_m u(p_x) &= \tfrac14 \left[(1-x) + (x-x) + (w-x) + (y-x)\right] = -\gl_m' x \\
  \gD_m u(p_y) &= \tfrac14 \left[(x-y) + (w-y) + (z-y) + (0-y)\right] = -\gl_m' y \\
  \gD_m u(p_z) &= \tfrac14 \left[(y-z) + (w-z) + (z-z) + (0-z)\right] = -\gl_m' z \\
  \gD_m u(p_w) &= \tfrac16 \left[2(x-w) + 2(y-w) + 2(z-w)\right] = -\gl_m' w,
\end{align*}
which can be rewritten, using $\gl_m = 4\gl_m'$, as
\linenopar
\begin{align*}
  (4-\gl_m)x &= 1 + x + y + w
  && (4-\gl_m)y = x + z + w \\
  (4-\gl_m)z &= y + z + w
  && (4-\gl_m)w = \tfrac43(x + y + z).
\end{align*}
For now, we suppress the dependence on $m$ for convenience and denote $\gl = \gl_m$. Solving for \gl, we obtain
\linenopar
\begin{subequations} \label{eqn:def-alpha-beta-gamma-phi}
\begin{align}
  x = \ga(\gl) 
     &:= (96-109\gl+33\gl^2-3\gl^3) /\gf(\gl), \\
  y = \gb(\gl) 
     &:= (16 - 3\gl)(3-\gl) /\gf(\gl), \\
  z = \gg(\gl) 
     &:= (36-7\gl) /\gf(\gl), \\
  w = \gr(\gl) 
     &:= 4 (5-\gl)(3-\gl) /\gf(\gl), \\
  \text{where}\q\gf(\gl) &:= 3(5-\gl)(3-\gl)(4-6\gl + \gl^2),
\end{align}
\end{subequations}
and we see that the forbidden eigenvalues are $3,5,3 \pm \sqrt 5$.

\begin{figure}
  \centering
  \scalebox{0.80}{\includegraphics{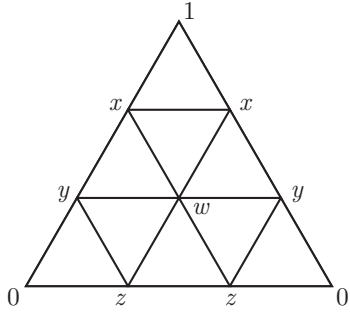}}
  \caption{The eigenfunction extension on one $(m-1)$-cell of $S\negsp[2]G_3$ to $m$-cells. The values on $(m-1)$-cell are $1,0,0$.}
  \label{fig:SG3-100}
\end{figure}

\begin{figure}
  \centering
  \scalebox{0.80}{\includegraphics{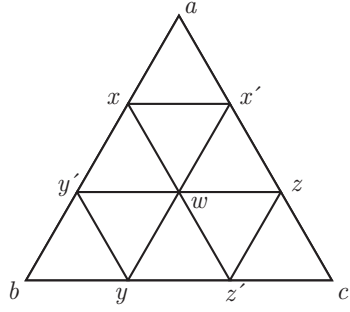}}
  \caption{The labeling for a general eigenfunction extension, from one $(m-1)$-cell to six $m$-cells. The values on the boundary of the $(m-1)$-cell are $a, b, c$.}
  \label{fig:SG3-abc}
\end{figure}

For a general function on $S\negsp[2]G_3$, we extend the eigenfunction using the labeling indicated in Figure \ref{fig:SG3-abc}, as follows:
\linenopar
\begin{align}
  x &= a \ga(\gl) + b \gb(\gl) + c \gg(\gl)
  & x' &= a \ga(\gl) + c \gb(\gl) + b \gg(\gl)
    \label{eqn:general-eigf-extension-system} \\
  y &= b \ga(\gl) + c \gb(\gl) + a \gg(\gl)
  & y' &= b \ga(\gl) + a \gb(\gl) + c \gg(\gl) \notag \\
  z &= c \ga(\gl) + a \gb(\gl) + b \gg(\gl)
  & z' &= c \ga(\gl) + b \gb(\gl) + a \gg(\gl) \notag \\
  w &= (a + b + c) \gr(\gl). \notag
\end{align}

The eigenfunction extension matrix for $S\negsp[2]G_3$  corresponding to
$F_0$ is
\linenopar
\begin{equation*}
  A_0(\gl) =
  \left[\begin{array}{ccc}
    1 & 0 & 0 \\
    \ga(\gl) & \gb(\gl) & \gg(\gl) \\
    \ga(\gl) & \gg(\gl) & \gb(\gl)
  \end{array}\right],
\end{equation*}
where we have $\ga(\gl),\gb(\gl),\gg(\gl),\gf(\gl)$ as before,
that is $A_0(\gl) u|_{V_0} = u |_{F_0 V_0}$.

\subsection{Eigenfunctions of the Laplacian on $S\negsp[2]G_3$}

\begin{figure}
  \centering
  \scalebox{0.80}{\includegraphics{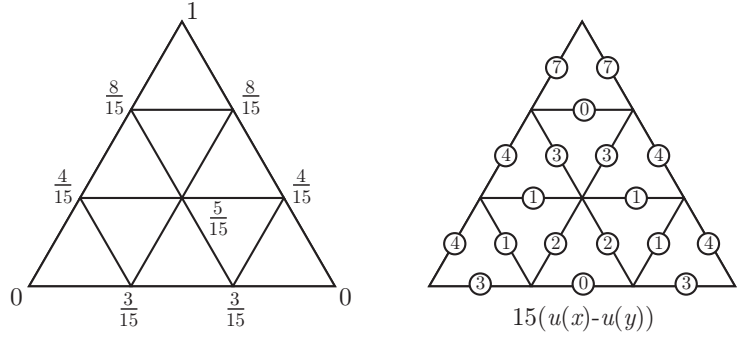}}
  \caption{The harmonic extension of $u$ on $S\negsp[2]G_3$, where $u|_{V_0} = [1,0,0]$.}
  \label{fig:SG3-100-harmonic}
\end{figure}

Let $u$ be a function taking values $1,0,0$ on $V_0$. The harmonic extension $\tilde u$ on $V_1$ corresponds to taking $\gl=0$ in the system (\ref{eqn:def-alpha-beta-gamma-phi}) above, so that
\linenopar
\begin{align*}
  x = \frac{8}{15},
  \q y=\frac{4}{15},
  \q z=\frac{3}{15},
  \q w=\frac{5}{15},
\end{align*}
and we have Figure \ref{fig:SG3-100-harmonic}. Following \cite[\S1.3]{Str06}, the energy renormalization constant computed by
\linenopar
\begin{align*}
  E_1(\tilde u)
  &= \left(\frac{1}{15}\right)^2 \left(4\cdot 1^2 + 2\cdot2^2
   + 4\cdot3^2 + 4\cdot4^2 + 2\cdot7^2\right)
   = \frac{14}{15}.
\end{align*}
Since $E_0(u) = 1 + 1 = 2$,
\linenopar
\begin{align}\label{eqn:energy-renormalization}
  2 = E_0(u) = r^{-1} E_1(\tilde u) = r^{-1} \frac{14}{15} \q\implies\q r=\frac7{15}.
\end{align}
Thus, the normal derivatives on $S\negsp[2]G_3$ are computed by
\linenopar
\begin{equation}\label{eqn:normal-deriv-formula-for-SG3}
  \dn u(p)
    = \lim_{m \to \iy} \left(\frac{15}7\right)^m
      \sum_{p \nbr[m] y} (u(p)-u(y)).
\end{equation}

\begin{theorem}\label{thm:SG3-Laplacian}
  The pointwise formulation of the Laplacian on $S\negsp[2]G_3$ is
  \linenopar
  \begin{equation}\label{eqn:thm:SG3-Laplacian}
    \gD_\gm u(x) = 6 \lim_{m \to \iy} \left(\frac{90}7\right)^m \gD_m u(x_m),
  \end{equation}
  where $\{x_m\}$ is any sequence with $\lim x_m = x$ and $x_m \in V_m$.
\end{theorem}
\begin{proof}
  Following \cite[\S2.2]{Str06}, it is easy to compute
  \linenopar
  \begin{align*}
    \int h_{x_m}^{(m)}\,d\gm = \left\{\begin{array}{cl} \frac{2}{3\cdot 6^m} & \textit{if \;} deg(x_m)=4, \\
    \frac{1}{6^m} & \textit{if \;} deg(x_m)=6 \end{array}\right.
  \end{align*}
  since \gm is the standard (self-similar) measure on $S\negsp[2]G_3$.
  Thus, by (\ref{eqn:def-relationship-Laplacian-SG3-Gamma-m}),
  \linenopar
  \begin{equation}\label{eqn:thm:SG3-Laplacian}
    \gD_\gm u(x) = \lim_{m\to \iy} \left(\frac{15}7\right)^m \cdot 6^{m+1}\gD_m u(x)
           =6 \lim_{m \to \iy} \left(\frac{90}7\right)^m \gD_m u(x_m),
  \end{equation}
\end{proof}

Throughout, whenever there is discussion of an eigenvalue \gl, we
assume that we have been given the sequence $\{\gl_m\}_{m=0}^\iy$
which defines \gl via the decimation formula. Thus by
Theorem~\ref{thm:SG3-Laplacian},
\linenopar
\begin{equation}\label{eqn:eiv-seq-formula-for-SG3}
  \gl = 6\lim_{m \to \iy} \left(\frac{90}{7}\right)^m \gl_m' =\frac32 \lim_{m \to \iy} \left(\frac{90}{7}\right)^m \gl_m.
\end{equation}

\subsection{Computation of the normal derivatives}
\label{sec:Computation-of-the-normal-derivatives}

\begin{figure}
  \centering
  \scalebox{0.80}{\includegraphics{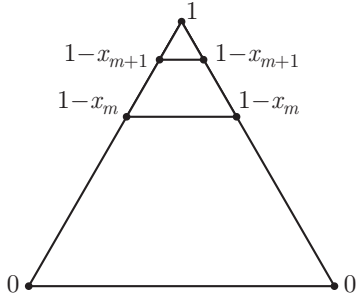}}
  \caption{The values of the eigenfunction $u$ on $S\negsp[2]G_3$,
  where $u|_{V_0} = [1,0,0]^T$. This figure shows a closeup of $u$
  near the point where it takes the value 1. By symmetry,
  we define $x_m := 1-u(F_0^mq_1)=1-u(F_0^mq_2)$.}
  \label{fig:SG3-100-side-seq}
\end{figure}

\begin{theorem}\label{thm:normal-derivs}
  Let $-\gD u = \gl u$ on $S\negsp[2]G_3$, where $u$ is defined on $V_0$
  by $u(q_0) = 1$, $u(q_1) = 0$, and $u(q_2) = 0$. Define
  \linenopar
  \begin{equation}\label{eqn:def-tau}
    \gt(\gl) := \frac{2\gl}{3\gl_0} \prod_{j=1}^\iy
    \frac{(1-\frac{\gl_j}4)(1-\frac{\gl_j}6)}{1-\frac32\gl_j+\frac{\gl_j^2}4}.
  \end{equation}
  Then the normal derivatives of $u$ are
  \linenopar
  \begin{subequations}\label{eqn:norm-derivs}
  \begin{align}
    \dn u(q_0) &= \frac{4-\gl_0}2 \gt(\gl), \q\text{and}
      \label{eqn:norm-derivs-p0}\\
    \dn u(q_1) &= \dn u(q_2) = - \gt(\gl).
      \label{eqn:norm-derivs-p1}
  \end{align}
  \end{subequations}
\end{theorem}
\begin{proof}
  To obtain \eqref{eqn:norm-derivs-p0} we
  need the values $u(F_0^mq_1)=u(F_0^mq_2)=1-x_m$ as depicted in Figure
  \ref{fig:SG3-100-side-seq}.

  Claim: $u(F_0^mq_1)=u(F_0^mq_2)=1-x_m$, where $x_0=1$ and
  \linenopar
  \begin{align}
    x_{m+1} - \frac{\gl_{m+1}}4 &= \frac{(4-\gl_{m+1})(6-\gl_{m+1})\gl_{m+1}}{(4-6\gl_{m+1}+\gl_{m+1}^2)\gl_m}
               \left(x_m - \frac{\gl_m}4\right).
    \label{eqn:thm:SG3-100-side-seq}
  \end{align}
  \begin{proof}[Proof of claim]
    By \eqref{eqn:general-eigf-extension-system}, if $b=c$, then
    $x=x'$. Then from $u(q_1)=u(q_2)=0$, we will have $u(F_0^m
    q_1)=u(F_0^m q_2)$ for all $m$, by induction. Define
    $x_m := 1-u(F_0^m q_1)$, $m \geq 0$. From $u(q_1)=0$, we have
    $x_0 = 1$. Now we show \eqref{eqn:thm:SG3-100-side-seq}
    holds.

    Denote $\gd(\gl) := \gb(\gl) + \gg(\gl)$, so that $\gd(\gl) =
    (14-3\gl)(6-\gl)/\gf(\gl)$, where $\gf(\gl)$ is as in
    \eqref{eqn:def-alpha-beta-gamma-phi}.
    Using \eqref{eqn:general-eigf-extension-system}, we have
    the matrix equation
    \linenopar
    \begin{align*}
      A_0(\gl) \left[\begin{array}{c} 1 \\ 1-x_m \\ 1-x_m \end{array}\right]
      = \left[\begin{array}{c} 1 \\ 1-x_{m+1} \\ 1-x_{m+1} \end{array}\right]
    \end{align*}
    gives
    \linenopar
    \begin{align*}
      x_{m+1}
      &= 1 - \ga(\gl_{m+1}) - \gd(\gl_{m+1}) + \gd(\gl_{m+1}) x_m \\
      &= - \frac{\gl_{m+1}(5-\gl_{m+1})}{4-6\gl_{m+1}+\gl_{m+1}^2}
         + \frac{(14-3\gl_{m+1})(6-\gl_{m+1})}{\gf(\gl_{m+1})} x_m.
    \end{align*}
    From the decimation relation \cite[(2.12)]{DrSt07}, we have the identity
    \linenopar
    \begin{align*}
      \frac{3(5-\gl_{m+1})(4-\gl_{m+1})(3-\gl_{m+1})\gl_{m+1}}{(14-3\gl_{m+1})\gl_{m}} = 1,
    \end{align*}
    so that
    \linenopar
    \begin{align*}
      \gd(\gl_{m+1}) &= \gd(\gl_{m+1}) \frac{3(5-\gl_{m+1})(4-\gl_{m+1})(3-\gl_{m+1})\gl_{m+1}}{(14-3\gl_{m+1})\gl_{m}} \\
        &= \frac{(4-\gl_{m+1})(6-\gl_{m+1})\gl_{m+1}}{(4-6\gl_{m+1}+\gl_{m+1}^2)\gl_m}.
    \end{align*}
    We would like to see $x_{m+1} - f(\gl_{m+1}) = \gd(\gl_{m+1})(x_m -
    f(\gl_m))$ for some function $f$, which is equivalent to
    \linenopar
    \begin{align*}
      (4-6\gl_{m+1}+\gl_{m+1}^2) \frac{f(\gl_{m+1})}{\gl_{m+1}}
      = \frac{f(\gl_m)}{\gl_m} (4-\gl_{m+1})(6-\gl_{m+1}) - (5-\gl_{m+1}).
    \end{align*}
    Let $f(x) = xg(x)$ and this can be rewritten
    \linenopar
    \begin{align*}
      (4-6\gl_{m+1}+\gl_{m+1}^2) g(\gl_{m+1})
      = g(\gl_m) (24-10\gl_{m+1}+\gl_{m+1}^2) - (5-\gl_{m+1}),
    \end{align*}
    which is easily seen to be true for the constant function
    $g(x) = \frac14$. Hence we may define $f(x) = \frac x4$, to
    obtain
    \linenopar
    \begin{align*}
      x_{m+1} - \frac{\gl_{m+1}}4 &= \gd(\gl_{m+1}) \left(x_{m} - \frac{\gl_{m}}4\right).
      \qedhere
    \end{align*}
  \end{proof}

  Now we compute $\dn u(q_0)$ using \eqref{eqn:thm:SG3-100-side-seq}
  to obtain
  \linenopar
  \begin{align*}
    x_m - \frac{\gl_m}4 &= \left(1-\frac{\gl_0}4\right) \frac{\gl_m}{\gl_0}
      \prod_{j=1}^m \frac{(4-\gl_j)(6-\gl_j)}{4-6\gl_j+\gl_j^2} \\
    x_m &= \frac{4-\gl_0}{4\gl_0} \left(\frac{\gl_0}{4-\gl_0} +
      \prod_{j=1}^m \frac{(4-\gl_j)(6-\gl_j)}{4-6\gl_j+\gl_j^2}\right) \gl_m.
  \end{align*}
  Since $u(q_0)=1$, we apply \eqref{eqn:normal-deriv-formula-for-SG3} to
  compute
  \linenopar
  \begin{align*}
    \dn u(q_0)
    &= \lim_{m \to \iy} \left(\frac{15}{7}\right)^m \left(\frac{6}{6}\right)^m
      (2u(q_0) - 2(1-x_m)) \\
    &= \frac{4-\gl_0}{2\gl_0} \lim_{m \to \iy} \left(\frac{90}{7}\right)^m \gl_m
      \left(\frac{\gl_0}{6^m(4-\gl_0)} +
      \prod_{j=1}^m \frac{(4-\gl_j)(6-\gl_j)}{6(4-6\gl_j+\gl_j^2)}\right) \\
    &= \frac{4-\gl_0}{2\gl_0} \left(\frac23 \gl\right)
      \left(0 + \prod_{j=1}^\iy \frac{(4-\gl_j)(6-\gl_j)}{6(4-6\gl_j+\gl_j^2)}\right),
  \end{align*}
  which is equivalent to the result.

  \pgap

  Now we compute the normal derivatives \eqref{eqn:norm-derivs-p0}. To obtain
  $\dn u(q_1) = \dn u(q_2)$, we don't actually need the
  values $u(F_1^mq_0)$ and $u(F_1^mq_2)$ as depicted in Figure
  \ref{fig:SG3-010-side-seq}.
  Instead, it suffices to only compute their sum, since by
  \eqref{eqn:normal-deriv-formula-for-SG3}, one has
  \linenopar
  \begin{align}\label{eqn:normal-deriv-at-0pt-for-SG3}
    \dn u(q_1) = \dn u(q_2)
      &= - \lim_{m \to \iy} \left(\frac{15}7\right)^m \left(u(F_1^m q_0)+u(F_1^m q_2)\right).
  \end{align}
  To exploit this symmetry accordingly, define
  \linenopar
  \begin{align*}
    y_m := u(F_0^mq_1), \q
    z_m := u(F_0^mq_2), \q\text{and}\q
    s_m := y_m + z_m.
  \end{align*}

  Claim: the sequence $\{s_m\}_{m=0}^\iy$ is given recurrently by $s_0=1$ and
  \linenopar
  \begin{align}\label{eqn:recursion-for-sm}
    s_{m+1} &=
    \frac{(14-3\gl_{m+1})(6-\gl_{m+1})}{\gf(\gl_{m+1})} s_m.
  \end{align}
  \begin{proof}[Proof of claim]
    As indicated in Figure \ref{fig:SG3-010-side-seq}, dihedral
    symmetry allows us to continue using the same matrix $A_0(\gl)$
    for computations, as long as we use $[0,1,0]^T$ for the new
    boundary data.

    It is clear that $s_0 = 1 + 0$ from the values on $V_0$. Then
    using the notation $\gd(\gl) = \ga(\gl) + \gb(\gl)$
    as above, the matrix equation
    \linenopar
    \begin{align*}
      A_0(\gl) \left[\begin{array}{c} 0 \\ y_m \\ z_m \end{array}\right]
      = \left[\begin{array}{c} 0 \\ \gb(\gl_{m+1})y_m + \gg(\gl_{m+1})z_m \\ \gg(\gl_{m+1})y_m + \gb(\gl_{m+1})z_m \end{array}\right]
      = \left[\begin{array}{c} 0 \\ y_{m+1} \\ z_{m+1} \end{array}\right]
    \end{align*}
    gives $s_{m+1} = y_{m+1} + z_{m+1} = \gd(\gl_{m+1})s_m$ immediately.
  \end{proof}
  Since \eqref{eqn:recursion-for-sm} gives
  \linenopar
  \begin{align*}
    s_m = \prod_{j=1}^m \gd(\gl_j) s_0
    = \frac{\gl_m}{\gl_0} \prod_{j=1}^m \frac{(4-\gl_j)(6-\gl_j)}{4-6\gl_j+\gl_j^2},
  \end{align*}
  and $u(q_1)=0$, the normal derivative is
  \linenopar
  \begin{align*}
    \dn u(q_1)
    &= \lim_{m \to \iy} \left(\frac{15}{7}\right)^m \left(\frac{6}{6}\right)^m
      (2u(q_1) - s_m) \\
    &= -\frac{1}{\gl_0} \lim_{m \to \iy} \left(\frac{90}{7}\right)^m\gl_m
      \prod_{j=1}^m \frac{(4-\gl_j)(6-\gl_j)}{6(4-6\gl_j+\gl_j^2)} \\
    &= -\frac{1}{\gl_0} \left(\frac23 \gl\right)
      \prod_{j=1}^\iy \frac{(4-\gl_j)(6-\gl_j)}{6(4-6\gl_j+\gl_j^2)}.
    \qedhere
  \end{align*}
\end{proof}

\begin{figure}
  \centering
  \scalebox{0.80}{\includegraphics{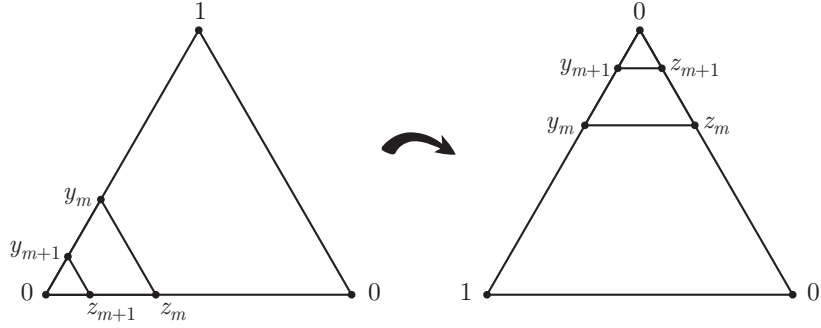}}
  \caption{The values of the eigenfunction $u$ on $S\negsp[2]G_3$,
  where $u|_{V_0} = [0,1,0]^T$. This figure shows a closeup of $u$
  near a point where it takes the value 0. See \eqref{eqn:normal-deriv-at-0pt-for-SG3} and the ensuing discussion.}
  \label{fig:SG3-010-side-seq}
\end{figure}

\subsection{The resolvent prekernel}

As in \eqref{eqn:Bobmatrix-preview}, let $B\parlaq{pq} := \sum_{K_j \ni q} \dn[K_j] \gy_\gl^{(p)}(q)$ for $p \in V_1 \setminus V_0$.

\begin{cor}\label{thm:Bpp}
  With $\gt(\gl)$ as in Thm. \ref{thm:normal-derivs} and $r = \frac{7}{15}$,
  \linenopar
  \begin{align*}
    \q B\parlaq{pq} = - r^{-1} \gt(\gl), \; \text{for } p \nbr[1] q
    \q \text{ and } \q
    B\parlaq{pp} &=
    \begin{cases}
      \tfrac32 r^{-1} (4-\gl_0) \gt(\gl), & p = p_0 \\
      r^{-1} (4-\gl_0) \gt(\gl), & p \neq p_0.
    \end{cases}
  \end{align*}
\end{cor}
\begin{proof}
  We are now working on $V_1$, so each term has a leading factor
  of $r^{-1}$. Whenever $p \nbr[1] q$, there is just one term $\dn u(q)
  = - \gt(\gl)$ in the sum; the other corner of the triangle is
  ignored and everything outside this 1-cell is 0. When $p=q$,
  then there is a sum of terms $\dn u(p) = \frac{4-\gl_0}2
  \gt(\gl)$. At the center point $p_0$, there are three such
  terms; at every other point there are only two.
\end{proof}

The matrix $B\parlaq{pq}$ is
\begin{equation}\label{eqn:inverse-resolvent-explicitly}
  \tfrac{15}{7}\gt(\gl)
  \left[\begin{array}{ccccccc}
    \scalebox{0.80}{\ensuremath{\tfrac32(4-\gl_0)}} & -1 & -1 & -1 & -1 & -1 & -1 \\
    -1 & \scalebox{0.80}{\ensuremath{(4-\gl_0)}} & -1 &  0 &  0 &  0 & -1 \\
    -1 & -1 & \scalebox{0.80}{\ensuremath{(4-\gl_0)}} & -1 &  0 &  0 &  0 \\
    -1 &  0 & -1 & \scalebox{0.80}{\ensuremath{(4-\gl_0)}} & -1 &  0 &  0 \\
    -1 &  0 &  0 & -1 & \scalebox{0.80}{\ensuremath{(4-\gl_0)}} & -1 &  0 \\
    -1 &  0 &  0 &  0 & -1 & \scalebox{0.80}{\ensuremath{(4-\gl_0)}} & -1 \\
    -1 & -1 &  0 &  0 &  0 & -1 & \scalebox{0.80}{\ensuremath{(4-\gl_0)}} \\
  \end{array}\right]
\end{equation}

\begin{defn}\label{def:resolvent prekernel}
  Define the \emph{resolvent prekernel} by $G\parlam := (B\parlam)^{-1}$.
\end{defn}

Our final result may be obtained by brutal and direct computation.

\begin{theorem}\label{thm:Gpq}
  The resolvent prekernel $G\parlam$ is given by
  \begin{equation}\label{eqn:thm:Gpq}
    \frac{14}{15(6-\gl)\gt(\gl)\gf(\gl)}
    \left[\begin{array}{ccccccc}
      (2-\gl)\gk_1  & \gk_1 & \gk_1 & \gk_1 & \gk_1 & \gk_1 & \gk_1 \\
      \gk_1 & \gk_2 & \gk_3 & \gk_4 & \gk_5 & \gk_4 & \gk_3 \\
      \gk_1 & \gk_3 & \gk_2 & \gk_3 & \gk_4 & \gk_5 & \gk_4 \\
      \gk_1 & \gk_4 & \gk_3 & \gk_2 & \gk_3 & \gk_4 & \gk_5 \\
      \gk_1 & \gk_5 & \gk_4 & \gk_3 & \gk_2 & \gk_3 & \gk_4 \\
      \gk_1 & \gk_4 & \gk_5 & \gk_4 & \gk_3 & \gk_2 & \gk_3 \\
      \gk_1 & \gk_3 & \gk_4 & \gk_5 & \gk_4 & \gk_3 & \gk_2
    \end{array}\right],
  \end{equation}
  where
  \begin{align*}
    \gk_1 &= (3-\gl)(5-\gl)(6-\gl), \\
    \gk_2 &= 201 - 300\gl + \tfrac{269}{2}\gl^2 - 24\gl^3 + \tfrac32\gl^4, \\
    \gk_3 &= 87 - 75\gl + 19\gl^2 - \tfrac32\gl^3, \\
    \gk_4 &= 57 - 24\gl + \tfrac52\gl^2, \text{ and}\\
    \gk_5 &= 51 - 15\gl - \gl^2.
  \end{align*}
  In particular, $G\parlam$ is symmetric and invertible with
  determinant
  \begin{equation}\label{eqn:det-G}
    \det G\parlam =\left(\frac{7}{15}\right)^7
      \frac{6(4-6\gl+\gl^2)}{(6-\gl)\gf(\gl)^2 \gt(\gl)}.
  \end{equation}
\end{theorem}

\bibliographystyle{alpha}
\bibliography{resolvent-kernel}
\nocite{*} 

\end{document}